\documentclass[12pt]{amsart}
\pdfoutput=1 
\usepackage{amssymb,amsthm, amsmath,mathtools,calligra,mathrsfs}
\usepackage{dsfont}
\usepackage{tikz}
\usetikzlibrary[matrix,arrows,calc,intersections, decorations.pathreplacing,decorations.markings]
\tikzset{
  on each segment/.style={
    decorate,
    decoration={
      show path construction,
      moveto code={},
      lineto code={
        \path [#1]
        (\tikzinputsegmentfirst) -- (\tikzinputsegmentlast);
      },
      curveto code={
        \path [#1] (\tikzinputsegmentfirst)
        .. controls
        (\tikzinputsegmentsupporta) and (\tikzinputsegmentsupportb)
        ..
        (\tikzinputsegmentlast);
      },
      closepath code={
        \path [#1]
        (\tikzinputsegmentfirst) -- (\tikzinputsegmentlast);
      },
    },
  },
  mid arrow/.style={postaction={decorate,decoration={
        markings,
        mark=at position .5 with {\arrow[#1]{stealth}}
      }}},
}

\usepackage{enumitem}

\usepackage[left=3.5cm,top=3.5cm,right=3.5cm]{geometry}

\usepackage{mathptmx} 
\usepackage[scaled]{helvet}
\usepackage{courier}
\usepackage{microtype}

\definecolor{cadmiumgreen}{rgb}{0.0, 0.42, 0.24}
\usepackage[
colorlinks, citecolor=cadmiumgreen,
pagebackref,
pdfauthor={Andreas Gross, Farbod Shokrieh}, 
pdfstartview ={FitV},
]{hyperref}
\hypersetup{
pdftitle={Tautological cycles on tropical Jacobians},
}

\usepackage[
alphabetic,
backrefs,
msc-links,
nobysame,
lite,
]{amsrefs} 

\theoremstyle{plain}
\newtheorem{thm}{Theorem}[section]
\theoremstyle{definition}
\newtheorem{defn}[thm]{Definition}
\theoremstyle{definition}
\newtheorem{example}[thm]{Example}
\theoremstyle{plain}
\newtheorem{lemma}[thm]{Lemma}
\theoremstyle{plain}
\newtheorem{cor}[thm]{Corollary}
\theoremstyle{remark}
\newtheorem{rem}[thm]{Remark}
\theoremstyle{plain}
\newtheorem{prop}[thm]{Proposition}
\theoremstyle{plain}

\theoremstyle{remark}
\newtheorem{notation}[thm]{Notation}
\theoremstyle{remark}
\newtheorem{convention}[thm]{Convention}
\theoremstyle{remark}

\theoremstyle{plain}

\theoremstyle{definition}

\numberwithin{equation}{section}

\let\bigstarvar\bigstar
\let\bigstar\relax
\DeclareMathOperator*{\bigstar}{\bigstarvar}

\DeclareMathOperator{\Hom}{Hom}
\DeclareMathOperator{\End}{End}
\DeclareMathOperator{\Homs}{\mathscr{H}\text{\kern -3pt {\calligra\large om}}\,}

\DeclareMathOperator{\divv}{div}

\DeclareMathOperator{\CDiv}{CDiv}

\DeclareMathOperator{\Pic}{Pic}

\DeclareMathOperator{\cyc}{cyc}

\DeclareMathOperator{\coker}{coker}

\newcommand{\reg}{{\operatorname{reg}}}
\newcommand{\Aff}{\operatorname{Aff}}

\newcommand{\AJ}{\Phi}
\newcommand{\Jac}{\operatorname{Jac}}

\newcommand{\vf}[1]{\delta_{#1} }

\newcommand{\R}{{\mathds R}}
\newcommand{\Rbar}{{\overline\R}}

\newcommand{\Q}{{\mathds Q}}
\newcommand{\Z}{{\mathds Z}}
\newcommand{\N}{{\mathds N}}

\newcommand{\mL}{{\mathcal L}}
\newcommand{\mM}{{\mathcal M}}

\newcommand{\tW}{{\widetilde W}}

\newcommand\boundaryless{boundaryless}

\begin{document}

\title{Tautological cycles on tropical Jacobians}

\author{Andreas Gross}
\address{Colorado State University \\ Fort Collins, CO 80523}
\email{\href{mailto:andreas.gross@colostate.edu}{andreas.gross@colostate.edu}}

\author{Farbod Shokrieh}
\address{University of Washington \\ Seattle, WA 98195}
\email{\href{mailto:farbod@uw.edu}{farbod@uw.edu}}


\subjclass[2010]{
\href{https://mathscinet.ams.org/msc/msc2010.html?t=14T05}{14T05},
\href{https://mathscinet.ams.org/msc/msc2010.html?t=14H40}{14H40},
\href{https://mathscinet.ams.org/msc/msc2010.html?t=14H42}{14H42},
\href{https://mathscinet.ams.org/msc/msc2010.html?t=14H51}{14H51}
}


\date{\today}

\begin{abstract}
The classical Poincar\'e formula relates the rational homology classes of tautological cycles on a Jacobian to powers of the class of Riemann theta divisor. We prove a tropical analogue of this formula. Along the way, we prove several foundational results about real tori with integral structures (and, therefore, tropical abelian varieties). For example, we prove a tropical version of the Appell-Humbert theorem. We also study various notions of equivalences between tropical cycles and their relation to one another.
\end{abstract}

\maketitle

\setcounter{tocdepth}{1}
\tableofcontents

\section{Introduction}

\renewcommand*{\thethm}{\Alph{thm}}

\subsection{Background}

Let $C$ be a compact Riemann surface of genus $g$. Its Jacobian variety $J$ has a number of natural subvarieties $\tW_d$ for $d \geq 0$, defined up to translation. The origin is denoted by $\tW_0$, the image of the Abel-Jacobi map is denoted by $\tW_1$, and $\tW_d = \tW_{d-1} + \tW_1$ is the image of higher symmetric powers of $C$. One can intersect these subvarieties, add again, pull back or push down under multiplication by integers, and so on. This provides a large supply of algebraic {\em tautological cycles}, which live naturally in $J$.

By the Riemann-Roch or Jacobi inversion theorem, one has $\tW_g = J$. Riemann's theorem states that $\tW_{g-1}$ is a shift of the Riemann theta divisor $\Theta$ (see, e.g., \cite[page 338]{GH78} , \cite[Chapter 1, \S5]{ACGH}, or \cite[Theorem 11.2.4]{bila}). The classical Poincar\'e formula gives a refinement of Riemann's theorem (see, e.g., \cite[page 350]{GH78}, \cite[Chapter 1, \S5]{ACGH}, or \cite[\S11.2]{bila}). It states that for $0 \leq d \leq g$ the classes of $\tW_d$ coincides with $\Theta^{g-d}$ in rational homology (up to the multiplicative constant $1/(g-d)!$). In other words, the subalgebra of tautological cycles in $H_*(J;\Q)$ is generated by the class of Riemann theta divisor.
There are also versions of the Poincar\'e formula over a general field. For example, Lieberman proves `Weil cohomology' statement (see \cite[Remark 2A13]{Kleiman}), and Mattuck proves a `numerical equivalence' statement (see \cite[\S2]{Mattuck}).

\subsection{Our contribution}

Our main goal in this paper is to prove a tropical analogue of the Poincar\'e formula. Let $\Gamma$ be a compact connected metric graph of genus $g$. Following \cite{MZjacobians}, one associates to $\Gamma$ a $g$-dimensional polarized real torus $\Jac(\Gamma)$, called its tropical Jacobian. There is also a well-behaved theory of divisors, ranks, Abel-Jacobi maps, and Picard groups for metric graphs (\cites{MZjacobians,GathmannKerber,BakerNorine}). We denote the tropical Abel-Jacobi morphism by $\Phi \colon \Gamma_d \rightarrow \Jac(\Gamma)$, which is well-defined up to a translation. Here $\Gamma_d$ denotes the set of all unordered $d$-tuples of points of $\Gamma$. The image $\tW_d =\Phi(\Gamma_d)$ is a polyhedral subset of $\Jac(\Gamma)$ of pure dimension $d$. Exactly as in the classical situation $\tW_d$ may be identified with the {\em effective locus} $W_d \subseteq \Pic^d(\Gamma)$ via the Abel-Jacobi map. In \cite{MZjacobians} one also finds the notion of Riemann theta divisor $\Theta$ on $\Jac(\Gamma)$, which is closely related to the theory of Voronoi polytopes of lattices. The polyhedral subsets $\tW_d$ and $\Theta$ of $\Jac(\Gamma)$ support {\em tropical fundamental cycles} $[\tW_d]$ and $[\Theta]$ (see \S\ref{sec:geometric cycles}). 
Recently, the notions of tropical homology, cohomology, and the cycle class map have been developed in \cite{TropHomology} and further studied in \cite{AF1}.

\begin{thm}[= Theorem \ref{thm:Poincare formula} and Corollary \ref{cor:poincare up to numerical equivalence}]
\label{introthm:Poincare formula}

For every $0\leq i\leq g$, we have the equality
\[
[\tW_d]= \frac{[\Theta]^{g-d}}{(g-d)!}
\]
on $\Jac(\Gamma)$ modulo tropical homological equivalence. Moreover, the equality also holds modulo numerical equivalence.
\end{thm}

Our proof further provides explicit descriptions of the classes of $\tW_d$ and $\Theta^{g-d}$ in tropical homology in terms of the combinatorics of the metric graph $\Gamma$ (see \S\ref{subsec:cycle classes of tautological cycles} and \S\ref{subsec:powers of theta}).

 The Poincar\'e formula has several interesting, but immediate, consequences.

\begin{cor}[= Corollaries \ref{cor:Riemman's Theorem}, \ref{cor:degree of tWd times tWg-d}, and \ref{cor:Theta to the
g}]
\label{introcor:consequences}

\begin{enumerate}
\item[]
\item[(a)] There exists a unique $\mu\in \Pic^{g-1}(\Gamma)$ such that $[W_{g-1}]=[\Theta]+\mu \ . $

\item [(b)] The effective tropical $0$-cycle obtained from the stable intersection of $[\tW_d]$ and $[\tW_{g-d}]$ has degree $g \choose d$.

\item [(c)] The tropical $0$-cycle $[\Theta]^g$ has degree $g!$.

\end{enumerate}

\end{cor}

We note that part (a) is a tropical version of Riemann's Theorem and has already been proven by Mikhalkin and Zharkov \cite{MZjacobians} using other combinatorial techniques.
The special case $d=1$ of part (b) can also be found in \cite{MZjacobians} in the context of the Jacobi inversion theorem, where again the proof is direct and combinatorial. This was essential in the development of break divisors in their paper. Part (c) classically follows from the geometric Riemann-Roch theorem for abelian varieties (see, e.g., \cite[Theorem 3.6.3]{bila}).

Building up to the proof of the Poincar\'e formula we also prove several foundational results about real tori with integral structures (and, therefore, about tropical abelian varieties) some of which had been used implicitly in previous work on the subject. Most notably, we prove the following tropical version of the Appell--Humbert Theorem:

\begin{thm}[= Theorem \ref{thm:Appell-Humbert}]

\label{introthm:Appell-Humbert}

Every tropical line bundle on a real torus $N_\R/\Lambda$ corresponds to a pair $(E,l)$ of a symmetric form $E$ on $N_\R$ with $E(N,\Lambda)\subseteq \Z$ and a morphism $l\in \Hom(N_\R,\R)$. Two such pairs $(E,l)$ and $(E',l')$ define the same line bundle if and only if $E=E'$ and $(l-l')(N)\subseteq \Z$.
\end{thm}

We also study the relationship between various notions of equivalence of tropical cycles. For example, we prove the following statement.

\begin{thm}[= Propositions \ref{prop:algebraic equivalence implies homological equivalence} and \ref{prop:homological equivalence implies numerical equivalence}]
\label{thm:equivalences}
Algebraic equivalence implies homological equivalence, and homological equivalence implies numerical equivalence on real tori admitting a `spanning curve'.
\end{thm}

\subsection{Further directions}

We believe our Poincar\'e formula is a first step in proving the following ambitious conjecture in tropical Brill-Noether theory. Let $W^r_d \subseteq \Pic^d(\Gamma)$ denote the locus of divisor classes of {\em degree} $d$ and {\em rank} at least $r$ (see, e.g., \cite{CDPR, LPN}).

\medskip

\noindent {\bf Conjecture.}
Assume $\rho = g - (r+1)(g-d+r) \geq 0$. Then there exists a canonical tropical subvariety $Z^r_d \subseteq W^r_d$ of pure dimension $\rho$ such that 
\[
[Z^r_d] = \left(\prod_{i=0}^r \frac{i!}{(g-d+r+i)!}\right) [\Theta]^{g-\rho} \ .
\]
modulo tropical homological equivalence.

\medskip

Note that our Theorem~\ref{introthm:Poincare formula} precisely establishes this conjecture in the case $r=0$, in which case $ W^0_d = W_d $ is pure-dimensional by \cite[Theorem 8.3]{semibreak} (see also Theorem~\ref{thm:pure dimensionality}) and $Z^0_d = W_d$. We also remark that a less precise version of this conjecture is posed as a question in \cite[Question 6.2]{Pflueger}.

As stated above, it follows from the Poincar\'e formula that the subring of tautological cycles in rational homology is too simple to provide interesting invariants. A celebrated result of Ceresa \cite{Ceresa} implies that for a generic curve $C$, the class of $W_d$ is {\em not} proportional to the class of $\Theta$ modulo algebraic equivalence. Beauville in \cite{Beauville} (see also \cite{Polishchuk, Marini, Moonen}) has studied results about algebraic equivalence. We believe that the tautological subring of the ring of tropical cycles modulo algebraic equivalence is an interesting object to study. For example, one might hope that this ring is generated by the classes of the $W_d$'s for $1 \leq d \leq g-1$. We remark that a tropical version of Ceresa's result has already been established by Zharkov in \cite{ZharkovMinusC}.

As stated in Theorem \ref{thm:equivalences}, homological equivalence implies numerical equivalence on tropical abelian varieties. We expect this to be true in general on any tropical manifold. 

 In analogy with Grothendieck's `standard conjecture D' one might also hope that homological equivalence coincides with numerical equivalence, at least in the case of tropical abelian varieties. The analogous classical result has been established by Lieberman in \cite{Lieberman}.

\subsection{The structure of this paper}

In \S\S\ref{sec:rational polyhedral spaces}--\ref{sec:tropical Jacobians} we review the main objects and tools needed to proof the Poincar\'e formula, including rational polyhedral spaces, tropical cycles, tropical homology, and tropical Jacobians.

In \S\S\ref{sec:notions of equivalence}--\ref{sec:line bundles and Appell-Humbert}  we study tropical cycles, tropical homology, and line bundles on real tori. Our results here are of a more foundational nature, and include the Appell-Humbert Theorem. We also study various  notions of equivalences of tropical cycles and prove Theorem \ref{thm:equivalences}.

Finally, in \S\S\ref{sec:geometric cycles}--\ref{sec:the formula} we prove the Poincar\'e formula.
In \S\ref{sec:geometric cycles} we show that the set $\tW_i$ has a fundamental cycle. In \S\ref{sec:the formula} we give explicit expression for both the cycle classes of the $[\tW_i]$ and of powers of the theta divisor. Comparing these expression will finish the proof of Theorem \ref{introthm:Poincare formula}. The results summarized in Corollary \ref{introcor:consequences} will be direct consequences of the Poincar\'e formula.

\subsection*{Acknowledgements}
AG was supported by the ERC Starting Grant MOTZETA (project 306610) of the European Research Council (PI: Johannes Nicaise) during parts of this project.

\renewcommand*{\thethm}{\arabic{section}.\arabic{thm}}

\subsection*{Notation}
We will denote by $\N$ the natural numbers including $0$. For an Abelian group $A$ and a topological space $X$, we will denote by $A_X$ the constant sheaf on $X$ associated to $A$.

\section{Rational polyhedral spaces}
\label{sec:rational polyhedral spaces}

The tropical spaces studied in this paper are real tori with integral structures, compact tropical curves, and their Jacobians. They all live inside the category of \boundaryless{} rational polyhedral spaces. We quickly review their definition and refer to \cite{MZeigenwave,Lefschetz,AF1} for more details.

\subsection{\boundaryless{} rational polyhedral spaces}

A rational polyhedral set in $\R^n$ is a finite union of finite intersections of sets of the form 
\[
\{x\in \R^n\mid \langle m, x\rangle\leq a\} \ ,
\]
where $m\in (\Z^n)^*$, $a\in \R$, and $\langle\cdot,\cdot\rangle$ denotes the evaluation pairing.
Any such set $P$ comes with a sheaf $\Aff_P$ of \emph{integral affine} functions, which are precisely the continuous real-valued functions that are locally (on $P$) of the form $x\mapsto \langle m, x\rangle +a$ for some $m\in (\Z^n)^*$ and $a\in \R$. 

\begin{defn}
A \emph{\boundaryless{} rational polyhedral space} is a pair $(X,\Aff_X)$ consisting of a topological space $X$ and a sheaf of continuous real-valued functions $\Aff_X$ such that every point $x\in X$ has an open neighborhood $U$ such that there exists a rational polyhedral set $P$ in some $\R^n$, an open subset $V\subseteq P$, and a homeomorphism $f\colon U \to V$ that induces an isomorphism $f^{-1}(\Aff_P|_V)\cong \Aff_X|_U $ via pulling back functions. Such an isomorphism $f$ is called a \emph{chart} for $X$. A \boundaryless{} rational polyhedral space that is compact is called a \emph{closed rational polyhedral space}. The sections of $\Aff_X$ are called \emph{integral affine functions}.
\end{defn}

\begin{rem}
\label{rem:spaces with boundary}
In the literature (for example in \cite{Lefschetz, AF1}), the notion of rational polyhedral spaces is used for spaces that are locally isomorphic to open subsets of rational polyhedral sets in $\Rbar^n$, where $\Rbar=\R\cup\{\infty\}$. This introduces a notion of boundary, which is essential for many applications. For our purposes it is sufficient to consider spaces without boundary. A \boundaryless{} rational polyhedral space is precisely a rational polyhedral space without boundary.
\end{rem}

\begin{defn} \phantomsection
\begin{itemize}
\item[]
\item[(i)] A morphism of \boundaryless{} rational polyhedral spaces is a continuous map $f\colon X\to Y$ such that pullbacks of functions in $\Aff_Y$ are in $\Aff_X$. 
\item[(ii)] A morphism $f\colon X\to Y$ is called proper if it is a proper map of topological spaces, that is preimages of compact sets are compact. 
\end{itemize}
\end{defn}

\subsection{Real tori with integral structures}
Let $N$ be a lattice, and let $\Lambda\subseteq N_\R=N\otimes_\Z\R$ be a second lattice of full rank, that is such that the induced morphism $\Lambda_\R\to N_\R$ is an isomorphism. Clearly, $N_\R$ gets a well-defined rational polyhedral structure from any isomorphism $N\cong \Z^n$. The \emph{real torus (with integral structure)} associated to $N$ and $\Lambda$ is the quotient $X=N_\R/\Lambda$, with the sheaf of affine functions being the one induced by $N_\R$. More precisely, if $\pi\colon N_\R\to X$ denotes the quotient map, and $U\subseteq X$ is open, then $\phi\colon U\to \R$ is in $\Aff_{X}(U)$ if and only if $\phi\circ \pi\in \Aff_{N_\R}(\pi^{-1}U)$. Note that the integral affine structure on $X$ is induced by $N$ and \emph{not} by $\Lambda$.

The group law on a real torus $X$ makes it a group object in the category of \boundaryless{} rational polyhedral spaces. In particular, every $x\in X$ defines an automorphism via translation.
\begin{defn}
Let $X$ be a real torus and let $x\in X$. Then the \emph{translation by $x$} is the morphism
\[
t_x\colon X\to X, \;\;y\mapsto x+y \ .
\]
\end{defn}

\subsection{Tropical curves}\label{subsec:tropical curves} A \emph{tropical curve} is a purely $1$-dimensional \boundaryless{} rational polyhedral space. With this definition, the underlying space of a tropical curve $\Gamma$ is a topological graph. In particular, it has a finite set of vertices (branch points) $V(\Gamma)$ where $\Gamma$ does not locally look like an open interval in $\R$, and a set of open edges $E(\Gamma)$, which are the connected components of $\Gamma\setminus V(\Gamma)$. The closed edges of $\Gamma$ are the closures of its open edges and an open edge segment is a connected open subset of an open edge. A tropical curve is \emph{smooth} (see Figure \ref{fig:singular and smooth curve}) if every point has a neighborhood that is isomorphic to a neighborhood of the origin in a \emph{star-shaped set}, that is a set of the form 
\[
\bigcup_{0\leq i\leq n} \R_{\geq 0}e_i \subseteq \R^{n+1}/\R \mathbf 1 \ .
\]
Here $n>0$ will denote the valency of the point, we denote by $\mathbf 1$ the vector whose coordinates are all $1$, and $e_i$ denotes the $i$-th standard basis vector.

Using the integral structure on a compact tropical curve, one can assign lengths to its edges, thus defining a metric graph. Conversely, given a metric graph (a topological graph $\Gamma$ equipped with an inner metric), one can define $\Aff_\Gamma$ as the sheaf of \emph{harmonic} functions on $\Gamma$, that is the sheaf of functions whose sum of incoming slopes is $0$ at every point. In this way, one obtains a smooth tropical curve $(\Gamma,\Aff_\Gamma)$ (cf.\ \cite[Proposition 3.6]{MZjacobians}).

The \emph{genus} $g$ of a tropical curve $\Gamma$ is defined as its first Betti number, that is $g=\dim_\R H_1(\Gamma;\R)$.

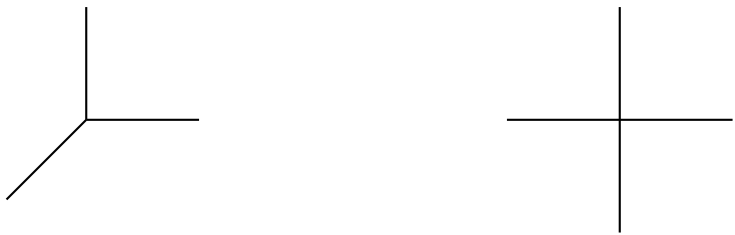
\begin{figure}
\centering
\begin{minipage}{.45\textwidth}
\centering
\begin{tikzpicture}[thick]
\clip (0,0) circle (1.5cm);
\draw (0,0)--(5,0);
\draw (0,0)--(0,5);
\draw (0,0)--(-5,-5);
\end{tikzpicture}
\end{minipage}
\quad
\begin{minipage}{.45\textwidth}
\centering
\begin{tikzpicture}[thick]
\clip (0,0) circle (1.5cm);
\draw (-5,0)--(5,0);
\draw (0,-5)--(0,5);
\end{tikzpicture}
\end{minipage}
\caption{Two tropical curves embedded in $\R^2$. The one to the left is smooth, the one to the right is not.}
\label{fig:singular and smooth curve}
\end{figure}

\begin{rem}
\label{rem:curves with boundary}
With our notion of tropical curves, the underlying topological graph is not allowed to have $1$-valent vertices. This can be resolved by working in the larger category of polyhedral spaces with boundary mentioned in Remark \ref{rem:spaces with boundary} and allowing neighborhoods of $\infty$ in $\Rbar$ as local models for the curves. In this way, tropical curves could have edges of infinite length that end in a $1$-valent vertex. But as we will note in Remark \ref{rem:reduction to boundaryless case}, the results of this paper are easily generalized to apply to compact and connected smooth tropical curves with boundary as well.
\end{rem}

\begin{example}
For any positive real number $j\in \R_{>0}$ the sublattice $\Z j$ of $\R=\Z_\R$ has full rank. Therefore, the quotient $\Gamma=\R/\Z j$, endowed with the integral affine structure induced by $\Z$, is a $1$-dimensional real torus. It is also a smooth tropical curve of genus $1$. Its unique edge is both open and closed and it is homeomorphic to the $1$-sphere. The length of this edge is given by $j$, which can be considered as the $j$-invariant of $\Gamma$ \cite{jinvariant}.
\end{example}

\begin{example}
\label{example:metric graph}
Consider the topological space $\Gamma$ obtained by gluing three intervals $[0,a]$, $[0,b]$, and $[0,c$] along their lower and upper bounds, respectively. Clearly, $\Gamma$ is a topological graph with three edges and two vertices. We can view the three intervals as rational polyhedral spaces, so on the interior of the edges of $\Gamma$ we have a notion of linearity. We can now define $\Aff_\Gamma$ as the sheaf of all continuous functions whose restrictions to the interiors of the intervals are linear, and such that the sum of the outgoing slopes is $0$ at the two vertices. With these choices, $\Gamma$ is the smooth tropical curve associated to the metric graph with three parallel edges of lengths $a$, $b$ and $c$. It is depicted in Figure \ref{fig:tropical curve}.
\end{example}

\begin{figure}[b]
\begin{tikzpicture}[font=\footnotesize,
dot/.style= {circle, fill=black, inner sep=.05cm},thick
]
\node [dot](upper) at (0,0) {};
\node [dot] (lower) at (2,0) {};
\coordinate(upperhelp) at (1,1);
\coordinate(lowerhelp)  at (1,-1);

\draw (upper)--node[pos=.5,auto] {$b$}(lower);
\draw  (upper) to [out=-120, in=-180] node[pos=1,auto,swap] {$c$} node [pos=.5](middle){} (lowerhelp);
\draw (lowerhelp) to [out=0, in=-60] (lower);
\draw  (upper) to [out=120, in=-180]  node[pos=1,auto] {$a$} (upperhelp);
\draw (upperhelp) to [out=0, in=60] (lower);

\draw [dashed] (lower) circle (.4cm);
\draw [dashed] (middle) circle (.4cm);

\begin{scope}[xshift=3.5cm]
\draw [dashed] (0,0) circle (0.4cm);
\clip (0,0) circle (0.4cm);
\node [dot] (vertex) at (0,0){};
\draw (0,0) --(-1,-1);
\draw (0,0)--(1,0);
\draw (0,0)--(0,1); 
\end{scope}

\begin{scope}[xshift=-.5cm,yshift=-1.7cm]
\coordinate (vertex2) at (0,0) {};
\draw [dashed] (0,0) circle (0.4cm);
\clip (0,0) circle (0.4cm);
\draw (-1,0)--(1,0);
\end{scope}

\draw (2.5,0)--  (3,0);
\draw ($(vertex2)!0.4!(middle)$) -- ($(vertex2)!0.6!(middle)$);
\end{tikzpicture}
\caption{A tropical curve of genus $2$ with its local charts and edge lengths.}
\label{fig:tropical curve}
\end{figure}
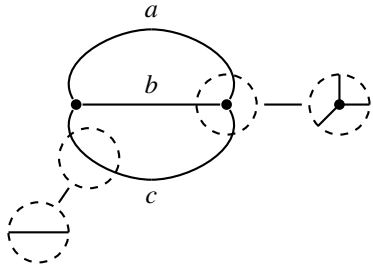

\subsection{Tropical manifolds}

We recall that every loop-free matroid $M$ on a ground set $E(M)$ has an associated tropical linear space $L_M$, which is a rational polyhedral set in $\R^{E(M)}/\R\mathbf 1 $. We will only consider very special linear spaces and therefore refrain from recalling their precise definition. For our purposes, it suffices to say that $\R^n$ is a tropical linear space for any $n$, and the $1$-dimensional tropical linear spaces are precisely the star-shaped sets appearing in the definition of smooth tropical curves in \S\ref{subsec:tropical curves}.

\begin{defn}
\label{def:tropical manifolds}
A \boundaryless{} rational polyhedral space $X$ is called a \emph{\boundaryless{} tropical manifold} if it can be covered by charts
\[
X\supseteq U\xrightarrow{\cong} V\subseteq L \ ,
\]
where $U$ is an open subset of $X$ and $V$ is an open subset of a tropical linear space $L$.
\end{defn}

Since both $\R^n$ and star-shaped sets are tropical linear spaces, it follows that real tori and smooth tropical curves are \boundaryless{} tropical manifolds.

\subsection{The cotangent sheaf}
\begin{defn} 
Let $X$ be a \boundaryless{} rational polyhedral space.
\begin{itemize}
\item[(i)] The quotient $\Aff_X/\R_X$ is called the \emph{cotangent sheaf} and is denoted by $\Omega^1_X$. 
\item[(ii)] The \emph{integral tangent space} at a point $x\in X$ is defined as $T^\Z_x X= \Hom(\Omega_{X,x},\Z)$. 
\item[(iii)] The \emph{tangent space} at a point $x\in X$ is defined as $T_x X= (T^\Z_x X)_\R\cong \Hom(\Omega_{X,x},\R)$. 
\end{itemize}
\end{defn}

\begin{example}
Let $X=N_\R/\Lambda$ be a real torus. Then $\Aff_X$ has no non-constant global sections because there is no globally defined non-constant integral affine function on $N_\R$ that is $\Lambda$-periodic. On the other hand, the quotient $\Aff_X/\R_X=\Omega^1_X$ is isomorphic to the constant sheaf $N_X$.
\end{example}

By definition, a morphism of \boundaryless{} rational polyhedral spaces $f\colon X\to Y$ induces a morphism $f^{-1}\Omega^1_Y\to \Omega^1_X$. Taking stalks and dualizing induces morphisms on tangent spaces $d_x f\colon T_x X \to T_{f(x)} Y$ for all $x\in X$ that map the integral tangent spaces on $X$ to the integral tangent spaces on $Y$.

\section{Tropical cycles and their tropical cycle classes}
\label{sec:tropical cycles}

We briefly recall the definitions of tropical cycles, tropical (co)homology, and the tropical cycle class map connecting the two. We closely follow \cite{AllerRau,FRIntersection,ShawIntersection} regarding tropical cycles and \cite{TropHomology, MZeigenwave, Lefschetz, AF1} regarding tropical (co)homology and the tropical cycle class map.

\subsection{Tropical cycles} \label{subsec:tropical cycles}

For a  \boundaryless{} rational polyhedral space $X$, let us denote by $X^\reg$ its open subset of points $x\in X$ that have a neighborhood isomorphic (as \boundaryless{} rational polyhedral spaces) to an open subset of $\R^n$ for some $n\in \N$. A tropical $k$-cycle is a function $A\colon X\to \Z$ such that its support $|A|= \overline{\{x\in X\mid A(x)\neq 0\}}$ is either empty or a purely $k$-dimensional polyhedral subset of $X$, $A$ is nonzero precisely on the set $|A|^\reg$, on which it is locally constant, and it satisfies the so-called \emph{balancing condition}. The latter is a local condition that is well-known for $X=\R^n$, to which the general case can be reduced. As we will only need it implicitly, we refer to \cite{AllerRau} for details. The sum of two tropical $k$-cycles on $X$, considered as a sum of $\Z$-valued functions, is not a tropical $k$-cycle again in general. However, there exists a unique tropical $k$-cycle on $X$ that agrees with the sum on the complement of an at most $(k-1)$-dimensional polyhedral subset of $X$. This makes the set $Z_k(X)$ into an Abelian group. A tropical cycle $A$ is said to be \emph{effective} if it is everywhere nonnegative.

If $f\colon X\to Y$ is a proper morphism of \boundaryless{} rational polyhedral spaces, it induces a \emph{push-forward} $f_* \colon Z_k(X)\to Z_k(Y)$ of tropical cycles. If $A\in Z_k(X)$ is a tropical cycle, then $f_*A$ will be zero outside of the subset $(f|A|)_k \subseteq f|A|$ where the local dimension of $f|A|$ is $k$. There exists a dense open subset $U\subseteq (f|A|)_k$ such that for each $y\in U$ the fiber $f^{-1}\{y\}$ is finite and contained in $|A|^\reg$, and for each such $y\in U$ we have
\[
f_*A(y)=\sum_{x\in f^{-1}\{y\}} |\coker d_x f| \ .
\]
Note that the finiteness of $\coker d_x f$ follows from the finiteness of the fiber over $y$. If $X$ is compact then one can take $Y$ to be a point. Identifying the tropical $0$-cycles on a point with $\Z$, the push-forward then defines a morphism $Z_0(X)\to \Z$. The image of a tropical $0$-cycle $A$ under this morphism is called the \emph{degree} of $A$, and it is denoted by $\int_X A$.

If $X$ and $Y$ are \boundaryless{} rational polyhedral spaces, and $A\in Z_k(X)$ and $B\in Z_l(X)$, then the \emph{cross product}
\[
A\times B\colon X\times Y\to \Z,\;\; (x,y)\mapsto A(x)\cdot B(x)
\]
of $A$ and $B$ is a tropical cycle again. 

A \emph{rational function} on a \boundaryless{} rational polyhedral space $X$ is a continuous function $\phi\colon X\to \R$ such that $\phi$ is piecewise affine  with integral slopes in every chart. As this is a local condition, rational functions define a sheaf $\mathcal \mM_X$ of Abelian groups. The group of \emph{tropical Cartier divisors} on $X$ is given by $\CDiv(X)=\Gamma(X, \mM_X/\Aff_X)$. For every $\phi\in \Gamma(X,\mM_X)$ we denote its image in $\CDiv(X)$ by $\divv(\phi)$, and refer to it as the associated \emph{principal divisor}. There exists natural bilinear map $\CDiv(X)\times Z_k(X)\to Z_{k-1}(X)$, the intersection pairing of divisors and tropical cycles.

Note that a \boundaryless{} rational polyhedral space $X$ does not automatically have natural fundamental cycle, that is there is no canonical element in $Z_*(X)$ in general. 
\begin{defn}
We will say that a \boundaryless{} rational polyhedral space \emph{$X$ has a fundamental cycle} if $X$ is pure-dimensional and the extension by $0$ of the constant function with value $1$ on $X^\reg$ defines a tropical cycle. In that case we will denote this tropical cycle by $[X]$, and refer to it as the {\em fundamental cycle} of $X$. We will say that a Cartier divisor $D\in \CDiv(X)$ on a tropical space $X$ with fundamental cycle is  \emph{effective}, if its \emph{associated Weil divisor} $[D]\coloneqq D\cdot [X]$ is effective.
\end{defn}

If $X$ is a tropical manifold then it has a fundamental cycle $[X]$, which is the unity of the \emph{tropical intersection product} on $Z_*(X)$. The tropical intersection product is compatible with intersections with Cartier divisors in the sense that
\[
D\cdot A= [D]\cdot A
\]
for every Cartier divisor $D\in \CDiv(X)$ and tropical cycle $A\in Z_*(X)$. Furthermore, the morphism
\[
\CDiv(X)\mapsto Z_{\dim(X)-1}(X), \;\; D\mapsto [D]
\]
is an isomorphism (see \cite[Corollary 4.9]{TropicalCocycles}).
If $X$ is locally isomorphic to open subsets of $\R^n$,  then a Cartier divisor $D\in \CDiv(X)$ is effective if and only if it is locally given by \emph{concave} rational functions. This follows from the fact that every tropical hypersurface of $\R^n$ is realizable. Here, a rational function is concave if it is the restriction of a concave rational function on $\R^n$ in sufficiently small local charts. Also note that concave functions appear rather than convex ones, because we are using the ``$\min$''-convention (see Remark \ref{rem:min convention}).

\subsection{Line bundles}
\label{subsec:line bundles}
A \emph{tropical line bundle} on a \boundaryless{} rational polyhedral space $X$ is an $\Aff_X$-torsor. More geometrically, it is a morphism $Y\to X$ of \boundaryless{} rational polyhedral spaces such that locally on $X$ there are trivializations $Y\cong  X\times \R$, where two such  trivializations are related via the translation by an integral affine function. More precisely, if two trivializations are defined over $U\subseteq X$, then the transition between them is of the form
\[
U\times \R \to U\times\R,\;\; (u, x)\mapsto (u, x+\phi(u))
\]
for some $\phi\in \Gamma(U,\Aff_X)$. The standard argument using \v{C}ech cohomology shows that the set of isomorphism classes of tropical line bundles on $X$ is in natural bijection to $H^1(X, \Aff_X)$. In particular, isomorphism classes of tropical line bundles form a group.
A rational section of a tropical line bundle $Y\to X$ is a continuous section that is given by a rational function in all trivializations. 
Exactly as in algebraic geometry, every tropical Cartier divisor $D$ on $X$ defines a line bundle $\mL(D)$ on $X$ that comes with a canonical rational section. This defines a bijection between $\CDiv(X)$ and isomorphism classes of pairs $(\mL,s)$ of a tropical line bundle $\mL$ on $X$ and a rational section $s$ of $\mL$.

\subsection{Homology and cohomology} \label{subsec:tropical homology}
Let $X$ be a \boundaryless{} rational polyhedral space. To define the tropical homology and cohomology groups, we need sheaves $\Omega^p_X$ of tropical $p$-forms for $p>0$. On the open subset $X^\reg$ it is clear that we would like $\Omega^p_X$ to be isomorphic to $\bigwedge^p \Omega^1_X$. However, this is not a suitable definition globally because in general $\bigwedge^p \Omega^1_X$ can be nonzero even for $p>\dim(X)$ (see \cite[Example 2.9]{AF1}).  One thus defines $\Omega^p_X$ as the image of the natural map
\[
\bigwedge\nolimits^p\Omega^1_X\to \iota_*\left(\bigwedge\nolimits^p\Omega^1_X\vert^{\vphantom{p}}_{X^\reg}\right) \ ,
\]
where $\iota\colon X^\reg\to X$ is the inclusion.

The singular tropical homology groups are defined similar to the integral singular homology groups, but with different coefficients. More precisely, there is a coarsest stratification of $X$ such that the restrictions of the constructible sheaf $\Omega^1_X$ is locally constant on all the strata, and only singular simplices are allowed that respect this stratification in the sense that each of their open faces is mapped into a single stratum. The $(p,q)$-th chain group is then defined as
\[
C_{p,q}(X)= \bigoplus_{\sigma\colon \Delta^q\to X \text{ allowable}}\Hom(\Omega^p_X,\Z_{\sigma(\Delta^q)}) \ ,
\] 
where $\Delta^q$ denotes the standard $q$-simplex, the sum runs over all $q$-simplices respecting the stratification, and $\Z_{\sigma(\Delta^q)}$ denotes the constant sheaf associated to $\Z$ on $\sigma(\Delta^q)$.
With the usual boundary operators this defines chain complexes $C_{p,\bullet}$ and the tropical homology groups which are defined as $H_{p,q}(X)=H_q(C_{p,\bullet}(X))$. 

Dualizing (over $\Z$) the chain complexes $C_{p,\bullet}(X)$ yields cochain complexes $C^{p,\bullet}(X)$ whose cohomology are the tropical cohomology groups $H^{p,q}(X)=H^q(C^{p,\bullet}(X))$. There is a natural isomorphism 
\[
H^{p,q}(X) \cong H^q(X,\Omega^p_X) \ .
\]

\subsection{The first Chern class map}
The quotient map $d\colon\Aff_X\to \Omega^1_X$ of sheaves on a \boundaryless{} rational polyhedral space induces a morphism 
\begin{equation*}
c_1\coloneqq H^1(d)\colon H^1(X,\Aff_X)\to H^1(X,\Omega^1_X) \cong H^{1,1}(X) 
\end{equation*}
called the \emph{first Chern class map} from the group of all tropical line bundles on $X$ to the $(1,1)$-tropical cohomology group of $X$. Using the first Chern class map, any divisor $D\in \CDiv(X)$ has an associated $(1,1)$-cohomology class $c_1(\mL(D))$.

\subsection{The tropical cycle class map}

Exactly as in algebraic geometry, there is a tropical cycle class map that assigns a class in tropical homology to every tropical cycle. More precisely, on any closed rational polyhedral space $X$, there exist morphisms
\[
\cyc\colon Z_k(X)\to H_{k,k}(X) 
\]
for every $k\in \N$.
We will only need an explicit description of the tropical cycle class map for $1$-dimensional tropical cycles, that is when $k=1$. If $A\in Z_1(X)$, then its support $\vert A\vert $ is a compact (not necessarily smooth) tropical curve. For each open edge $e$ of $\vert A\vert$ choose a generator $\eta_e\in T_x\vert A\vert$ for some $x\in e$. By taking parallel transports of $\eta_e$ along $e$ we actually obtain a generator for all $T_y\vert A\vert$ with $y\in e$. Therefore, $\eta_e$ defines a morphism $\Omega^1_{\vert A \vert} \to \Z_e$ (recall that $\Z_e$ denotes the constant sheaf on $e$ associated to $\Z$), which can be uniquely extended to a morphism $\Omega^1_{\vert A\vert } \to \Z_{\overline e}$. Precomposing with the morphism $\Omega^1_{X}\to \Omega^1_{\vert A\vert}$ defined by the inclusion $\vert A\vert \to X$, one obtains a morphism $\eta_{\overline e}\in \Hom(\Omega^1_X ,\Z_{\overline e})$. To complete the construction, one has to choose a homeomorphism $\gamma_{\overline e}\colon \Delta^1\to \overline e$ that parametrizes $\overline e$ in the direction specified by $\eta_e$. Let us denote the element in $C_{1,1}(X)$ defined by $\gamma_{\overline e}$ and $\eta_{\overline e}$ by $\gamma_{\overline e}\otimes \eta_{\overline e}$. Then $\cyc(A)$ is represented by the cycle
\[
\sum_e A(e)\cdot \gamma_{\overline e}\otimes \eta_{\overline e}\in C_{1,1}(X)  \ ,
\]
where the sum runs over all open edges of $\vert A\vert$ and $A(e)$ denotes the weight of the tropical $1$-cycle $A$ on $e$.

\begin{example}
Let $\Gamma$ be the graph from Example \ref{example:metric graph}, and denote its edges by $e_1$, $e_2$, and $e_3$. Let $v$ and $w$ be the vertices of $\Gamma$ and orient all edges from $v$ to $w$. Let $\eta_i$ be the primitive tangent direction on $e_i$ in the chosen direction. Then $\cyc[\Gamma]$ is represented by the $(1,1)$-chain
\[
\gamma_1\otimes \eta_1+ \gamma_2\otimes \eta_2 + \gamma_3\otimes \eta_3 \ ,
\]
where $\gamma_i$ is any path that parametrizes $e_i$ from $v$ to $w$. This is indeed a cycle. Its boundary is given by
\[
w\otimes (\eta_1+\eta_2+\eta_3) -v\otimes (\eta_1+\eta_2+\eta_3) \ ,
\]
which vanishes: locally at $v$ (respectively at $w$), the graph $\Gamma$ looks like the star-shaped set depicted to the right in Figure \ref{fig:singular and smooth curve}, and the vectors $\eta_i$ are the (negatives of the) primitive generators of the rays of the star-shaped set. Since these sum to $0$, the boundary is $0$. 
\end{example}

\subsection{Identities in tropical homology}

In \cite{AF1} we studied various operations on tropical homology and cohomology and showed how to carry over identities known for singular homology to the tropical setting. For example, there are pull-backs of cohomology classes and push-forwards of homology classes along morphisms of \boundaryless{} rational polyhedral spaces, there is a cup product ``$\smile$'' on tropical cohomology and a cap product ``$\frown$'' that makes the tropical homology groups a module over the tropical cohomology ring. There also are cross products ``$\times$'' of both homology and cohomology classes. We will refer the reader to \cite{AF1} for the details regarding these operations. For the reader's convenience, we have summarized the most important identities for the tropical cycle class map in the following theorem:
\begin{thm}[{\cite{AF1}}]
\label{thm:cycle class map commutes with operations}
Let $X$, $Y$, and $Z$ be closed rational polyhedral spaces, let $f\colon X\to Z$ be a proper morphism, let $A\in Z_*(X)$, $B\in Z_*(Y)$ and $D\in \CDiv(X)$. Then we have
\begin{align*}
\cyc(f_*A) \;&=\qquad f_*\cyc(A)\ ,\\
\cyc(A\times B)&= \quad \cyc(A)\times \cyc(B)\ \text{, and}  \\
\cyc(D\cdot A)\;&= \;c_1(\mL(D))\frown \cyc(A) \ .
\end{align*}
\end{thm}

If $X$ is a closed rational polyhedral space, then the morphism from $X$ to a point defines a morphism $H_{0,0}(X)\to \Z$ by identifying the $(0,0)$-tropical homology group of a point with $\Z$. The image of a tropical cycle $\alpha\in H_{0,0}(X)$ is called the \emph{degree of $\alpha$} and denoted by $\int_X\alpha$. It is a direct consequence of the first equation in Theorem \ref{thm:cycle class map commutes with operations} that 
\[
\int_X A=\int_X \cyc(A)
\]
for every $A\in Z_0(X)$.

If $X$ is a closed tropical manifold, then homology and cohomology are dual to each other, in the sense that the morphism
\[
H^{*,*}(X)\to H_{*,*}(X),\;\; c\mapsto c\frown \cyc[X]
\]
is an isomorphism \cite{Lefschetz,AF1}. In this context one says that $c$ is \emph{Poincar\'e dual} to $c\frown \cyc[X]$. Poincar\'e duality allows to define an \emph{intersection product} for tropical homology classes on a closed tropical manifold $X$. More precisely, if $\alpha,\beta\in H_{*,*}(X)$, and $c\in H^{*,*}(X)$ is Poincar\'e dual to $\alpha$, then one defines
\[
\alpha\cdot \beta\coloneqq c\frown \beta.
\]

\begin{rem}
\label{rem:min convention}
Both the intersection pairing between tropical Cartier divisors and tropical cycles, and the tropical cycle class map are not entirely free of choices. The intersection pairing depends on whether one measures incoming or outgoing slopes. When measuring incoming slopes, concave functions define effective principal divisors, whereas when measuring outgoing slopes, convex functions define effective principal divisors. Since minima of linear functions are concave, and maxima of linear functions are convex, one speaks of the ``$\min$''- and ``$\max$''-conventions, respectively. The cycle class map, on the other hands, depends on a consistent choice of isomorphisms 
\[
\bigwedge\nolimits^k N\xrightarrow{\cong} H_k(N_\R,N_\R\setminus\{0\};\Z)
\] 
for any lattice $N$ of any rank $k$ (see \cite[\S 5]{AF1}).

If one wants Theorem \ref{thm:cycle class map commutes with operations} to hold, one has to make the choices involved in the definitions of the intersection pairing and the cycle class map consistently. In other words, the choice of either ``$\min$''- or ``$\max$''-convention will determine the sign of the cycle class map. 
\textit{In this paper, we will choose the ``$\min$''-convention}, because it makes the formulas in \S\ref{sec:the formula} nicer, but the same formulas hold true in the ``$\max$''-convention after appropriately adjusting the sign.
\end{rem}

\section{Tropical Jacobians}
\label{sec:tropical Jacobians}
In this section we review the definition of tropical Jacobians, closely following \cite{MZjacobians}. Let $\Gamma$ be a compact and connected smooth tropical curve. We write $\Omega_\Z(\Gamma)\coloneqq H^0(\Gamma,\Omega_\Gamma^1)$ for the group of global integral $1$-forms, and $\Omega_\R(\Gamma)\coloneqq \Omega_\Z(\Gamma)\otimes_\Z\R$ for the group of (real) $1$-forms. A $1$-form on $\Gamma$ is completely determined by its restrictions to the edges of $\Gamma$, and these restrictions are constant and completely determined by a real number and an orientation of the edge: it will be of the form $rdx$, where $r\in \R$, and $x$ is the chart on the edge determined by the orientation. Extracting the data of its restrictions to the edges out of a $1$-form gives rise to a natural morphism $\Omega_\R(\Gamma)\to C_1(\Gamma;\R)$. Since the outgoing primitive direction vectors at any point of $\Gamma$ (in any chart around that point) sum to $0$, the chains in the image of $\Omega_\R(\Gamma)$ will in fact be $1$-cycles, that is they are mapped to $0$ by the boundary morphism. It is not hard to see that the induced map $\Omega_\R(\Gamma)\to H_1(\Gamma;\R)$ is an isomorphism.

\begin{rem}
\label{rem:flows}
Another way to think of the elements of $\Omega_\Z(\Gamma)$ is as integral flows. Given $\omega\in \Omega_\Z(\Gamma)$, we have already observed that the restriction $\omega|_e$ to an open edge $e$ is determined by a direction and a nonnegative integer. Conversely, a collection of directions and nonnegative integers for every edge in $\Gamma$ will define a global $1$-form if and only if this collection defines a flow.
\end{rem}

Global $1$-forms on $\Gamma$ can be integrated on singular $1$-chains in $\Gamma$. We obtain a pairing
\begin{equation}
\label{equ:integration pairing}
\Omega_\R(\Gamma)\times C_1(\Gamma;\R) \to \R ,\;\; (\omega, c) \mapsto \int_c\omega \ ,
\end{equation}
which can be shown to induce a morphism $H_1(X;\R)\to \Omega_\R(\Gamma)^*$. Together with the isomorphism $H_1(\Gamma;\R)\cong\Omega_\R(\Gamma)$ from above, we obtain a natural bilinear form $E$ on $H_1(\Gamma;\R)$, which can be described explicitly. Namely, for two $1$-cycles $c_1$ and $c_2$, the pairing $E(c_1,c_2)$ is the weighted length of the intersection of $c_1$ and $c_2$, where an oriented line segment occurring in $c_1$ and $c_2$ with weights $\lambda$ and $\mu$, respectively, contributes with weigh $\lambda\cdot \mu$. This bilinear form is clearly symmetric and positive definite. In particular, it is a perfect pairing, and hence the morphism $H_1(X;\R)\to \Omega_\R(\Gamma)^*$ we used to define it is an isomorphism. Via this isomorphism $H_1(\Gamma;\Z)$ becomes a sublattice of $\Omega_\R(\Gamma)^*$ of full rank, and the positive definite symmetric bilinear form $E$ induces a positive definite symmetric bilinear form $Q$ on $\Omega_\R(\Gamma)^*$. 
The full-rank sublattice of $\Omega_\R(\Gamma)^*$ that has integer pairings with the elements of $H_1(X,\Z)$ with respect to $Q$ is precisely $\Omega_\Z(\Gamma)^*$. 

\begin{defn}
\label{def:tropical Jacobian}
The \emph{tropical Jacobian} associated to the compact and connected smooth tropical curve $\Gamma$ is the pair consisting of the real torus
\[
\Jac(\Gamma) \coloneqq \Omega_\R(\Gamma)^*/ H_1(\Gamma;\Z) 
\]
and the bilinear form $Q$ that is defined on the universal cover $\Omega_\R(\Gamma)^*$ of $\Jac(\Gamma)$. 
\end{defn}

\begin{rem}
By the universal coefficient theorem, we also have an isomorphism $H^1(\Gamma;\R)\cong H_1(\Gamma;\R)^*$.  Together with the isomorphism $\Omega_\R(\Gamma)\cong H_1(\Gamma;\R)$ from above one obtains an isomorphism $H^1(\Gamma;\R)\cong \Omega_\R(\Gamma)^*$. It is therefore also possible to write the Jacobian  of $\Gamma$ as the quotient $H^1(\Gamma;\R)/H_1(\Gamma;\Z)$. 
\end{rem}

Now fix a base point $q\in \Gamma$. Given any other point $p\in \Gamma$ there is a path $\gamma_p$ connecting $q$ to $p$.
As any other path from $q$ to $p$ differs from $\gamma_p$ by an integral  $1$-cycle, the class of $\gamma_p$ in $\big(C_1(\Gamma;\Z)/B_1(\Gamma;\Z)\big)/H_1(\Gamma;\Z)$ is independent of the choice of $\gamma_p$. 
Here, $B_1(\Gamma;\Z)$ denotes the group of $1$-boundaries.
Using the pairing \eqref{equ:integration pairing}, we obtain an element in $\Jac(\Gamma)$ that only depends on the choice of $q$. This defines the \emph{Abel--Jacobi map}
\begin{equation*}
\AJ_q\colon \Gamma \to \Jac(\Gamma) \ .
\end{equation*} 

Let $p\in \Gamma$, and let $U$ be a sufficiently small connected open neighborhood of $p$. More precisely, $U$ should be connected and $U\setminus \{p\}$ should be disjoint from  $V(\Gamma)$. Then for every $p'\in U\setminus\{p\}$ there exists $r>0$ and a geodesic path $\gamma\colon [0,r]\to U$  from $p$ to $p'$. Let $e$ denote the unique open edge $e$ of $\Gamma$ containing $p'$, and let $\eta$ denote the primitive integral tangent vector on $e$ pointing from $p$ towards $p'$. If $x$ is any lift of $\AJ_q(p)$ to the universal cover $\Omega_\R(\Gamma)^*$, then by definition, $p'$ lifts to $x+r\cdot \delta$, where $\delta$ is given by
\[
\delta\colon \Omega_\R(\Gamma) \rightarrow \R ,\;\;\omega\mapsto \frac 1r \int_\gamma \omega= \langle\omega|_e,\eta\rangle \ .
\]	
If we identify $\Omega_\R(\Gamma)$ with flows on $\Gamma$ (as in Remark \ref{rem:flows}) then $\delta$ is the map assigning to a flow $\omega$ on $\Gamma$ its flow on $e$ in the direction specified by $\eta$. In particular, $\delta$ is integral, that is $\delta\in\Omega_\Z(\Gamma)^*$. This shows that $\AJ_q$ is, in fact, a morphism of \boundaryless{} rational polyhedral spaces, and that its action on the tangent space of $e$ is given by
\[
\delta=(d\AJ_q)(\eta)\ .
\]

\begin{example}
\label{example:Theta and W1}
Let $\Gamma$ be the smooth tropical curve associated to the metric graph that consists of two vertices which are connected by three edges of length $1$ (the graph of Example \ref{example:metric graph} with $a=b=c=1$). It is depicted to the left in Figure \ref{fig:Theta and W1}. We choose one of the vertices as the base point $q$ and orient the edges of $\Gamma$ such that one edge, call it $e_3$ is oriented towards $q$ and the other two edges, call them $e_1$ and $e_2$, are oriented away from $q$. The orientations define two simple closed loops $c_1$ and $c_2$ in $\Gamma$, where $c_i$ first follows $e_i$ and then $e_3$. These loops define a basis for $H_1(\Gamma;\R)$, and hence for $\Omega_\Z(\Gamma)$. Let $\delta_1,\delta_2\in \Omega_\Z(\Gamma)^*$ be the dual basis. Since the signed length of $c_i\cap c_j$ is $2$ if $i=j$ and $1$ if $i\neq j$, the injection $H_1(\Gamma;\Z)\to \Omega_\R(\Gamma)^*$ maps $c_1$ to $(2,1)$ and $c_2$ to $(1,2)$ in the coordinates defined by the basis $\delta_1,\delta_2$. If follows that 
\[
\Jac(\Gamma)= \raisebox{.75ex}{$\R^2$}\bigg/\raisebox{-.75ex}{$
\Z\begin{pmatrix}
1\\2
\end{pmatrix}
+\Z\begin{pmatrix}
2\\1
\end{pmatrix}$} \ ,
\]
where the integral structure is given by $\Z^2\subseteq \R^2$.

The Abel-Jacobi map sends $q$ to $0$ in this quotient. If $\gamma_1$ is the geodesic path along $e_1$ that starts at $q$, then
\[
(d\Phi_q)(\gamma_1(t))=  t\cdot \begin{pmatrix}
1\\0
\end{pmatrix} + H_1(\Gamma;\Z)
\]
for all $t\in [0,1]$
because the path from $q$ to $\gamma_1(t)$ along $e_1$ intersects $c_1$ in an edge segment of length $t$, and $c_2$ in a point (an edge segment of length $0$). Similarly, if $\gamma_2$ is a geodesic path along $e_2$, and $\gamma_3$ is a geodesic path along $e_3$, both starting at $q$, then
\begin{align*}
(d\Phi_q)(\gamma_2(t))&=  t\cdot \begin{pmatrix}
0\\1
\end{pmatrix} + H_1(\Gamma;\Z) \ \text{, and}\\
(d\Phi_q)(\gamma_3(t))&=  t\cdot \begin{pmatrix}
-1\\-1
\end{pmatrix} + H_1(\Gamma;\Z) 
\end{align*}
for all $t\in [0,1]$.
\end{example}

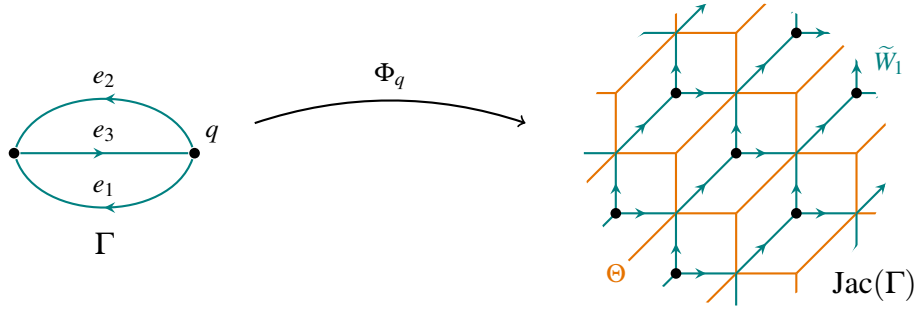
\begin{figure}
\begin{tikzpicture}[scale=0.8,font=\footnotesize,auto]

\draw [thick,->](-8,3.5) .. controls (-6.5,4) and (-5,4) .. node[above,pos=.5] {$\AJ_q$}(-3.5,3.5); 

\begin{scope}[xshift=-12cm,yshift=3cm, thick]
\node at (1.5,-1.5){\color{black}\normalsize $\Gamma$};
\node (left) at (0,0) [circle,fill=black,inner sep=0.05cm]{};
\node (right) at (3,0) [circle,fill=black,inner sep=0.05cm, label={[label distance=-.1cm]45:$q$} ]{}
  edge [teal, bend left=70,postaction={on each segment={mid arrow=teal}}] node[midway,swap,black]{$e_1$}(left)
  edge [teal, postaction={on each segment={mid arrow=teal}}, bend right=70] node[midway,swap,black]{$e_2$}(left);
  \draw [teal, postaction={on each segment={mid arrow=teal}}] (left)--node[midway,black]{$e_3$}(right);
\end{scope}

\begin{scope}[yshift=3cm]
\node at (-2,-2) {\color{orange!90!black}$\Theta$};
\node at (2.55,1.55) {\color{teal}$\widetilde W_1$};
\node at (2.3,-2.2) {\normalsize $\Jac(\Gamma)$};
\clip (0,0) circle [radius=2.53cm];
\foreach \i in {-4,...,4} {
	\foreach \j in {-4,...,4}  {
		\coordinate (p) at ($\i*(2,1)+\j*(1,2)$);			
		\draw [postaction={on each segment={mid arrow=teal}},teal,thick] (p)--++(1,0);
		\draw [postaction={on each segment={mid arrow=teal}},teal,thick] (p)--++(0,1);
		\draw [postaction={on each segment={mid arrow=teal}},teal,thick] ($(p)+(-1,-1)$)-- (p);
		\draw [orange!90!black,thick] ($(p)+(1,0)$) -- ($(p)+(1,1)$);
		\draw [orange!90!black,thick] ($(p)+(1,0)$) -- ($(p)+(2,0)$);
		\draw [orange!90!black,thick] ($(p)+(1,0)$) -- ($(p)+(0,-1)$);	
		\node at (p) [circle, fill=black, inner sep=0.05cm] {};	
	}}
\end{scope}
\end{tikzpicture}
\caption{
A tropical curve of genus $g$, the universal cover of its Jacobian, and the sets {\color{teal}$\widetilde W_1$} and {\color{orange!90!black}$\Theta$} lifted to the universal cover (see \S\ref{subsec:Theta divisor} for the definition of $\Theta$).}
\label{fig:Theta and W1}
\end{figure}

\section{Algebraic, homological, and numerical equivalence}
\label{sec:notions of equivalence}

In this section we study different notions of equivalence for tropical cycles on \boundaryless{} rational polyhedral spaces, with a focus on real tori.

\subsection{Algebraic equivalence}

Following \cite{ZharkovMinusC}, we make the following definition.

\begin{defn}
Let $X$ be a \boundaryless{} rational polyhedral space. Let $R_{\mathrm{alg}}$ be the subgroup of $Z_*(X)$ generated by tropical cycles of the form
\[
p_*(q^*(t_0-t_1)\cdot W) \ ,
\]
	where $W$ is a tropical cycle on $X\times \Gamma$ for some compact and connected smooth tropical curve $\Gamma$ containing the two points $t_0,t_1\in \Gamma$, and $p\colon X\times \Gamma \to X$ and $q\colon X\times \Gamma\to \Gamma$ are the natural projections. Note that because $\Gamma$ is smooth, the difference $t_0-t_1$ defines a tropical Cartier divisors on $\Gamma$ (see \S\ref{subsec:tropical cycles}) and tropical Cartier divisors can be  pulled-back along any morphism of \boundaryless{} rational polyhedral spaces.

 We say that two tropical cycles $A, B\in Z_*(X)$ are \emph{algebraically equivalent}, denoted by $A\sim_{\mathrm{alg}}B$, if their classes in $Z_*(X)/R_{\mathrm{alg}}$ coincide.
\end{defn}

\begin{prop}
\label{prop:algebraic equivalence defines an ideal}
Let $X$ be a \boundaryless{} tropical manifold, let $A,B,C\in Z_*(X)$ be tropical cycles on $X$, and assume that $A\sim_{\mathrm{alg}} B$. Then
\[
A\cdot C\sim_{\mathrm{alg}} B\cdot C \ .
\]
\end{prop}

\begin{proof}
By the definition of algebraic equivalence, we may assume that there exists a compact and connected smooth tropical curve $\Gamma$, two points $t_0,t_1\in \Gamma$, and a tropical cycle $W$ on $X\times \Gamma$ such that 
\[
A-B= p_*(q^*(t_0-t_1)\cdot W) \ ,
\]
where $p$ and $q$ denote the projection. Using the projection formula \cite[Theorem 8.3 (1)]{FRIntersection}, we see that
\[
A\cdot C- B\cdot C= (A-B)\cdot C = p_*(q^*(t_0-t_1)\cdot (W\cdot (C\times \Gamma))) \ .
\]
Applying the definition of algebraic equivalence with $W$ replaced  by $W\cdot (C\times \Gamma)$, we obtain that $A\cdot C$ and $B\cdot C$ are algebraically equivalent.
\end{proof}

\begin{defn}
Let $X=N_\R/\Lambda$ be a real torus. A \emph{spanning curve} for $X$ is a $1$-dimensional polyhedral subset $\Gamma\subseteq X$ such that there exists an effective tropical $1$-cycle on $X$ with support $\Gamma$, and such that the parallel transports to $0$ of the direction vectors of the edges of $\Gamma$ span $T_0 X\cong N_\R$. If such a curve exists, we say that $X$ \emph{admits a spanning curve}.
\end{defn}

\begin{prop}
Let $\Gamma$ be a compact and connected smooth tropical curve. Then its Jacobian $\Jac(\Gamma)$ admits a spanning curve.
\end{prop}

\begin{proof}
For any choice of base point $q\in \Gamma$, the image $\Phi_q(\Gamma)$ of $\Gamma$ under the Abel-Jacobi map is the support of the effective cycle $\Phi_{q*}[\Gamma]$. Using the explicit description given in \S\ref{sec:tropical Jacobians} of the tangent directions in $\Jac(\Gamma)$ of the images of the edges of $\Gamma$, it follows directly that $\Phi_q(\Gamma)$ is a spanning curve for $\Jac(\Gamma)$.
\end{proof}

\begin{prop}
\label{prop:translate is algebraically equivalent}
Let $X=N_\R/\Lambda$ be a real torus that admits a spanning curve $\Gamma$. Let $x\in X$, and recall that we denote by $t_x\colon X\to X$ the translation by $x$. Then for every tropical cycle $A\in Z_*(X)$ we have
\[
A \sim_{\mathrm{alg}} (t_{x})_*A \ .
\]
\end{prop}

\begin{proof}
By the assumptions on $\Gamma$, the point $x$ is in the subgroup of $X$ generated by the differences $y-y'$ for pairs $y,y'\in \Gamma$ contained in the same edge of $\Gamma$. Therefore, it suffices to show that $(t_{x})_*A \sim_{\mathrm{alg}} (t_{x'})_*A$ for any pair of points $x,x'$ contained in the same edge of $\Gamma$. Let $\Gamma_x$ be the component of $\Gamma$ containing $x$. Even though $\Gamma_x$ is not smooth, it still determines a metric graph $G$. After a choice of weights that makes $\Gamma$ into a tropical $1$-cycle, the metric graph $G$ is equipped with weights $m\colon E(G)\to \Z_{>0}$ induced by the weights on $\Gamma$. Let $\widetilde G$ be the metric graph obtained from $G$ replacing each edge $e$ of $G$ by an edge of length $\ell(e)/m(e)$, where $\ell(e)$ denotes the length of $e$ in the metric graph $G$. If $\widetilde \Gamma$ denotes the smooth tropical curve associated to the graph $\widetilde G$ (see \S\ref{subsec:tropical curves}), then there is a natural morphism $f\colon\widetilde \Gamma \to \vert\Gamma_x\vert$ of rational polyhedral spaces, which is a bijection of the underlying spaces. Let $t,t'\in \widetilde \Gamma$ be the unique points with $f(t)=x$ and $f(t')=x'$. Now let
\[
g\colon X\times \widetilde\Gamma \to X\times \widetilde \Gamma ,\;\; (x, s)\mapsto (x+f(s), s) \ ,
\]
and denote $W=g_*(A\times [\Gamma])\in Z_*(X\times \Gamma)$. By construction, if $p\colon X\times \widetilde \Gamma \to X$ and $q\colon X\times \widetilde \Gamma\to \widetilde \Gamma$ denote the projections, we have
\[
p_*(q^*(t)\cdot W)= (t_x)_*(A) \quad \text{and} \quad p_*(q^*(t')\cdot W)= (t_{x'})_*(A) \ ,
\]
finishing the proof.
\end{proof}

\subsection{Homological equivalence}

\begin{defn}
Let $X$ be a closed rational polyhedral space. We say that two tropical cycles $A$ and $B$ are \emph{homologically equivalent}, if $\cyc(A)=\cyc(B)$. 
\end{defn}

\begin{example}
\label{ex:homological equivalence on curves}
Let $\Gamma$ be a compact and connected smooth tropical curve. By definition, we have $H_{0,0}(\Gamma)\cong H_0(\Gamma;\Z)\cong\Z$. It follows that the degree morphism $H_{0,0}(\Gamma) \to \Z$ is an isomorphism. Therefore, the homological equivalence class of a tropical $0$-cycle is uniquely determined by its degree. Let $D\in \CDiv(\Gamma)$ be a Cartier divisor on $\Gamma$. By Theorem \ref{thm:cycle class map commutes with operations}, we have $\cyc[D]= c_1(\mL(D))\frown[\Gamma]$, and by Poincar\'e duality this implies that $c_1(\mL(D)=0$ if and only if $\cyc[D]=0$. By what we just saw, we have $\cyc[D]=0$ if and only if the degree of $D$ is $0$. We see that if $D'\in\CDiv(\Gamma)$ is another Cartier divisor, then $[D]$ and $[D']$ are homologically equivalent if and only if $c_1(\mL(D))=c_1(\mL(D'))$, which holds if and only if $D$ and $D'$ have the same degree.
\end{example}

\begin{prop}
\label{prop:algebraic equivalence implies homological equivalence}
Algebraic equivalence implies homological equivalence: if $A$ and $B$ are tropical cycles on a closed  rational polyhedral space $X$ with $A\sim_{\mathrm{alg}} B$, then $A\sim_{\mathrm{hom}} B$.
\end{prop}

\begin{proof}
By the definition of algebraic equivalence, we may assume that there exists a compact and connected smooth tropical curve $\Gamma$, two points $t_0,t_1\in \Gamma$, and a tropical cycle $W$ on $X\times \Gamma$ such that $A-B= p_*(q^*(t_0-t_1)\cdot W)$. Since $t_0-t_1$ has degree $0$, we have $c_1(\mL(t_0-t_1))=0$ (see Example \ref{ex:homological equivalence on curves}). Therefore, by Theorem \ref{thm:cycle class map commutes with operations}, we have
\[
\cyc(A)-\cyc(B)=\cyc(A-B)=p_*\big(q^*c_1(\mL(t_0-t_1))\frown \cyc(W)\big) = 0 \ ,
\]
finishing the poof.
\end{proof}

\begin{thm}
\label{thm:cycle map commutes with intersection product}
Let $X$ be a real torus admitting a spanning curve, and let $A,B\in Z_*(X)$ be tropical cycles. Then we have
\begin{equation*}
\cyc(A\cdot B) = \cyc(A)\cdot \cyc(B) \ .
\end{equation*}
\end{thm}

\begin{proof}
As both sides are bilinear in $A$ and $B$, we may assume that $A$ and $B$ are pure-dimensional, say of dimensions $k$ and $l$, respectively. By Proposition \ref{prop:translate is algebraically equivalent} and Proposition \ref{prop:algebraic equivalence implies homological equivalence}, we may replace $A$ by a general translate. Therefore, we can assume that $A$ and $B$ meet transversally,  that is that $\vert A\vert \cap \vert B\vert$ is either empty or of pure dimension $k+l-n$, where $n=\dim(X)$, and $(\vert A\vert \cap \vert B\vert)^\reg = \vert A\vert ^\reg \cap \vert B \vert ^\reg$. 

As explained in \cite[Remark 5.5]{AF1}, we can view $\cyc(A)$ as an element of the Borel-Moore homology group $H_{k,k}^{BM}(\vert A\vert ,X)$ with supports on $\vert A\vert$, and similarly 
\begin{align*}
\cyc(B)&\in H_{l,l}^{BM}(\vert A\vert, X) \quad\text{ and} \\
\cyc(A\cdot B)&\in H^{BM}_{k+l-n,k+l-n}(\vert A\cap  B\vert, X) \ .
\end{align*}
Using Verdier duality \cite[Theorem D]{AF1}, the cycle class $\cyc(A)$ is Poincar\'e dual to a cohomology class with support on $\vert A\vert$, that is to an element in $H_{\vert A\vert}^{n-k,n-k}(X)$. Therefore, the intersection product $\cyc(A)\cdot \cyc(B)$ is also represented by an element in $H_{k+l-n,k+l-n}^{BM}(\vert A\cap B\vert, X)$ and it suffices to prove the equality
\[
\cyc(A\cdot B) = \cyc(A)\cdot \cyc(B)
\]
in $H_{k+l-n,k+l-n}^{BM}(\vert A\cap B, X)$. For dimension reasons, both sides are uniquely determined by their restrictions to $H_{k+l-n,k+l-n}^{BM}(\vert A\cap B\vert\cap U, U)$, where $U$ is an open subset of $X$ with $U\cap \vert A\cap B\vert = \vert A\cap B\vert^\reg$ \cite[Lemma 4.8 (b)]{AF1}. Combining the facts that $V\mapsto H_{k+l-n,k+l-n}^{BM}(\vert A\cap B\vert\cap V, V)$ satisfies the sheaf axioms \cite[Lemma 4.8 (b)]{AF1},  $X$ is locally isomorphic to open subsets of $\R^n$, and $(\vert A\vert \cap \vert B\vert)^\reg = \vert A\vert ^\reg \cap \vert B \vert ^\reg$ allows us to further reduce to the case where $U=\R^n$ and $A$ and $B$ are linear subspaces of $\R^n$. In this case, there exist hyperplanes $H_1,\ldots, H_{n-k}$ and  $H'_1\ldots, H'_{n-l}$, and integers $a,b\in\Z$  such that 
\begin{align*}
A&= a\cdot H_1\cdots H_{n-k} \quad \text{and} \\
B&= b\cdot H'_1\cdots H_{n-l} \ .
\end{align*}
Let $\alpha\in H^{n-k,n-k}_{\vert A\vert}(X)$ be the Poincar\'e dual to $\cyc(A)$. Applying \cite[Proposition 5.12]{AF1} (see also \cite[Remark 5.13]{AF1}) yields
\begin{multline*}
\cyc(A\cdot B) =  \cyc\big((a\cdot H_1\cdots H_{n-k})\cdot (b\cdot H'_1\cdots H_{n-l})\cdot [X]\big) = \\ 
=\big(a \cdot c_1(\mL(H_1))\smile \ldots \smile c_1(\mL(H_{n-k}))\big)\frown \\
\frown \Big(\big(b\cdot c_1(\mL(H'_1))\smile \ldots\smile  c_1(\mL(H'_{n-l}))\big)\frown \cyc[X]\Big) = \\
=\alpha\frown \cyc(B) =
\cyc(A)\cdot \cyc(B) \ ,
\end{multline*}
where the last equality holds by the definition of the intersection product of tropical homology classes. This finishes the proof.
\end{proof}

\subsection{Numerical equivalence}

\begin{defn}
Let $X$ be a closed tropical manifold. Then two tropical cycles $A,B\in Z_*(X)$ on $X$ are \emph{numerically equivalent}, for which we write $A\sim_{\mathrm{num}} B$, if for every tropical cycle $C\in Z_*(X)$ on $X$ we have
\[
\int_X A\cdot C=\int_X B\cdot C \ .
\]
\end{defn}

\begin{prop}
\label{prop:homological equivalence implies numerical equivalence}
Let $X$ be a real torus admitting a spanning curve, and let $A,B\in Z_*(X)$ with $A\sim_{\mathrm{hom}} B$. Then $A\sim_{\mathrm{num}}B$.
\end{prop}

\begin{proof}
Let $C\in Z_*(X)$. By Theorem \ref{thm:cycle map commutes with intersection product}, we have
\begin{multline*}
\int_X A\cdot C =\int _X\cyc (A\cdot C)= \int_X \cyc(A) \cdot \cyc(C) =\\
=\int_X \cyc(B)\cdot \cyc(C) =\int_X\cyc(B\cdot C)= \int_X B\cdot C \ ,
\end{multline*}
from which the assertion follows.
\end{proof}

\section{Tropical homology of real tori} \label{sec:homology of Abelian varieties}

Let $X=N_\R/\Lambda$ be a real torus. Then the group law and the tropical cross product endow the tropical homology groups with the additional structure of the Pontryagin product:

\begin{defn}
Let $X$ be a real torus with group law $\mu\colon X\times X \to X$. The \emph{tropical Pontryagin product} is defined as the pairing
\[
(\alpha,\beta) \mapsto \alpha\star \beta \coloneqq \mu_*(\alpha\times \beta) \ ,
\]
where $\alpha$ and $\beta$ are either elements of $Z_*(X)$ or of $H_{*,*} (X)$. We thus obtain morphisms 
\begin{align*}
\star\colon Z_i(X)\otimes_\Z Z_k(X) &\to Z_{i+k}(X) \\
\star\colon H_{i,j}(X)\otimes_\Z H_{k,l}(X) &\to H_{i+k,j+l}(X) \\
\end{align*}
for all choices of natural numbers $i,j,k,l$. It is not hard to see that $\star$ makes $Z_*(X)$ into a graded abelian group, and $H_{*,*}(X)$ into a bigraded abelian group.
\end{defn}

\begin{prop}
\label{prop:Pontryagin product commutes with cycle map}
Let $X$ be a real torus. Then the tropical cycle class map respects Pontryagin products, that is the diagram 
\begin{center}
\begin{tikzpicture}[auto]
\matrix[matrix of math nodes, row sep= 5ex, column sep= 4em, text height=1.5ex, text depth= .25ex]{
|(ZXX)| Z_i(X)\otimes_\Z Z_j(X) 	&
|(ZX)|	Z_{i+j}(X)	\\
|(HXX)| H_{i,i}(X)\otimes_\Z H_{j,j}(X)	&	
|(HX)| H_{i+j,i+j}(X)	\\
};
\begin{scope}[->,font=\footnotesize]
\draw (ZXX) --node{$\star$} (ZX);
\draw (HXX) --node{$\star$} (HX);
\draw (ZXX) -- node{$\cyc\otimes \cyc$} (HXX);
\draw (ZX) -- node{$\cyc$} (HX);
\end{scope}
\end{tikzpicture}
\end{center}
is commutative for all $i,j\in \Sigma$
\end{prop}

\begin{proof}
Since the Pontryagin product is defined as the push-forward of a cross product, this follows immediately from the compatibility of the tropical cycle class map with cross products and push-forwards stated in Theorem \ref{thm:cycle class map commutes with operations}.
\end{proof}

For the real torus $X=N_\R/\Lambda$, we will now describe the group $H_{*,*}(X)$ and the Pontryagin product on it explicitly. First we note that the  sheaf $\Omega^1_X$ is the constant sheaf $M_X$ associated to the lattice $M=\Hom(N,\Z)$, and since $X^\reg=X$, we have $\Omega^k_X\cong \left(\bigwedge^k M\right)_X$ for all integers $k$. By definition of singular tropical homology, we thus have a canonical graded  isomorphism
\[
H_{*,*}(X)\cong H_*\left(X;\bigwedge\nolimits^* N\right)\cong H_*(X;\Z)  \otimes_\Z \bigwedge\nolimits^* N .
\] 
The restriction of the Pontryagin product to the first factor $H_*(X;\Z)\cong H_{0,*}(X)$ is precisely the classical Pontryagin product one obtains when one views $X$ is a topological group. But, as a topological group, $X$ is a product of $1$-spheres. So using the K\"unneth theorem one sees that $H_*(X;\Z)$ is isomorphic to $\bigwedge H_1(X;\Z)$. This is, in fact, an isomorphism of rings, the multiplication of $H_*(X;\Z)$ being the Pontryagin product. Finally, because $X$ is the quotient of its universal covering space $N_\R$ by the action of $\Lambda$, we obtain a natural isomorphism $H_1(X;\Z)\cong \Lambda$. If a tropical $1$-cycle in $H_1(X;\Z)$ is represented by a loop $\gamma\colon [0,1]\to X$ then the corresponding element of $\Lambda$ is given by $\widetilde \gamma (1)-\widetilde \gamma(0)$ for any lift $\widetilde \gamma \colon [0,1]\to N_\R$ of $\gamma$ to the universal cover. We obtain an isomorphism
\begin{equation}
\label{equ:homology of real torus}
H_{*,*}(X) \cong \bigwedge\nolimits^* \Lambda \otimes_\Z \bigwedge\nolimits^* N \ .
\end{equation}
It is straightforward to check that with this identification, the tropical Pontryagin product on $H_{*,*}(X)$ satisfies 
\[
(\alpha \otimes \omega) \star (\beta\otimes \xi) = (\alpha\wedge \beta) \otimes (\omega\wedge \xi) \ .
\]

By a similar argument, one obtains a description for the tropical cohomology of $X$ that is dual to the description of tropical homology in \eqref{equ:homology of real torus}. More precisely, one sees that
\begin{equation}
\label{equ:cohomology fo real torus}
H^{*,*}(X) \cong \bigwedge\nolimits^* \Lambda^* \otimes_\Z \bigwedge\nolimits^* M \ ,
\end{equation}
and that, with this identification, the tropical cup product on $H^{*,*}(X)$ satisfies
\[
(\alpha \otimes \omega) \smile (\beta\otimes \xi) = (\alpha\wedge \beta) \otimes (\omega\wedge \xi) \ .
\]

With the descriptions of the tropical homology and the tropical cohomology given in (\ref{equ:homology of real torus}) and (\ref{equ:cohomology fo real torus}), the tropical cap product can also be expressed explicitly. More precisely, we have
\begin{equation}
\label{equ:identity for cap product on real tori}
(\alpha \otimes \omega) \frown (\beta\otimes \xi) = (\alpha\,\lrcorner\, \beta) \otimes (\omega\,\lrcorner\, \xi) \ ,
\end{equation}
where ``$\lrcorner$'' denotes the \emph{interior product} on the exterior algebra.

In bidegree $(1,1)$ our description of the tropical cohomology of $X$ produces an isomorphism
\[
H^{1,1}(X) \cong \Lambda^* \otimes_\Z M \ .
\]
We can further identify the right side with $\Hom(\Lambda\otimes_\Z N,\Z)$, that is with bilinear forms on $N_\R$ that have integer values on $\Lambda\times N$.

\begin{convention}
\label{conv:convention on H11 and bilinear forms}
From now on we will always identify, according to the identifications in this section, the cohomology group $H^{1,1}(N_\R/\Lambda)$ with the group of bilinear forms on $N_\R$ that have integer values on $\Lambda\times N$.
\end{convention}

\section{Line bundles on real tori}
\label{sec:line bundles and Appell-Humbert}

\subsection{Factors of automorphy}
\label{subsec:factors of automorphy}

Let $N$ be a lattice, let $\Lambda\subseteq N_\R$ be a lattice of full rank, and let $X=N_\R/\Lambda$ be the real torus associated to $N$ and $\Lambda$. To describe the tropical line bundles on $X$ we recall from \S\ref{subsec:line bundles} that they form a group, canonically identified with $H^1(X, \Aff_X)$. Invoking the results from \cite[Appendix to \S2]{Mumford08}, together with the fact that the pull-back $\pi^{-1}\Aff_X \cong \Aff_{N_\R}$ along the quotient morphism $\pi\colon N_\R\to N_\R/\Lambda=X$ has trivial cohomology on $N_\R$, we obtain the identification
\begin{equation*}
H^1(X, \Aff_X) \cong H^1(\Lambda, \Gamma(N_\R,\Aff_{N_\R})) \ ,
\end{equation*}
where the right side is the first group cohomology group of $\Gamma(N_\R,\Aff_{N_\R})$, equipped with its natural $\Lambda$-action.
This is very much akin to the case of complex tori: an element of $H^1(\Lambda, \Gamma(N_\R,\Aff_{N_\R}))$ can be represented by a \emph{tropical factor of automorphy}, that is a family of integral affine functions indexed by $\Lambda$, that, if we represent it as a function $a\colon \Lambda\times N_\R\to \R$, satisfies
\begin{equation}
\label{equ:group action condition}
a(\lambda+\mu,x)=  a(\lambda, \mu+x) + a(\mu, x)
\end{equation}
for all $\mu,\lambda\in  \Lambda$ and $x\in N_\R$. Two factors of automorphy represent the same element of $H^1(\Lambda, \Gamma(N_\R,\Aff_{N_\R}))$ if and only if they differ by a factor of automorphy of the form
\begin{equation*}
(\lambda,x)\mapsto l(x+\lambda)-l(x)
\end{equation*}
for some integral affine function $l\in \Gamma(N_\R,\Aff_{N_\R})$, which happens if and only if they differ by a factor of automorphy of the form
\[
(\lambda,x)\mapsto m_\R(\lambda) \ ,
\]
where $m_{\R}$ is the $\R$-linear extension of a linear form $m\colon N\to \Z$.

Any factor of automorphy $a(-,-)$ defines a group action $\lambda . (x, b) = (x+\lambda, b+ a(\lambda,x))$ of $\Lambda$ on the trivial line bundle $N_\R\times \R$ on $N_\R$. The tropical line bundle on $X$ corresponding to $a(-,-)$ is the quotient $(N_\R\times \R)/\Lambda$.

\subsection{The Appell--Humbert Theorem}
\label{subsec:Appell-Humbert}

It is easy to check that for every morphism $l\in \Hom(\Lambda, \R)$ and every symmetric bilinear form $E$ on $N_\R$ with $E(\Lambda\times N )\subseteq \Z$, the family of integral affine functions on $N_\R$ defined by 
\begin{equation*}
a_{E,l} (\lambda,x)= l(\lambda) - E(\lambda, x) - \frac 12 E(\lambda,\lambda)
\end{equation*}
is a tropical factor of automorphy. We denote the associated tropical line bundle on $X$ by $\mL(E,l)$. The following proposition shows that the first Chern class recovers $E$ from $\mL(E,l)$.

\begin{prop}
\label{prop:chern class recovers symmetric form}
Let $E$ be a symmetric bilinear form on $N_\R$ with $E(\Lambda\times N)\subseteq \Z$, and let $l\in \Hom(\Lambda,\R)$. Then $c_1(\mL(E,l))=E$, where we identify $H^{1,1}(X)$ with the group of bilinear forms on $N_\R$ with integer values on $\Lambda\times N$ according to Convention \ref{conv:convention on H11 and bilinear forms}.
\end{prop}

\begin{proof}
Let $\mathfrak U=\{U_\alpha\}_\alpha$ be an open cover of $X$ such that each preimage $\pi^{-1}U_\alpha$ is a union of disjoint open subsets of $N_\R$ that map homeomorphically onto $U_\alpha$. For each $\alpha$, choose a continuous section $s_\alpha\colon U_\alpha\to \pi^{-1}U_\alpha$ of $\pi$. Furthermore, we choose a (necessarily non-continuous) section $s\colon X\to N_\R$ of $\pi$. By construction, the line bundle $\mL(E,l)$ is represented by the \v{C}ech cocycle
\begin{equation*}
(U_{\alpha,\beta}\ni x\mapsto a_{E,l}(s_\beta(x)-s_\alpha(x),s_\alpha(x))) \in \check C^1(\mathfrak U, \Aff_X) \ .
\end{equation*}
Note that $s_\beta-s_\alpha$ has values in $\Lambda$ and is therefore constant on the connected components of  $U_{\alpha,\beta}= U_\alpha\cap U_\beta$ by continuity. In particular, the functions $x\mapsto a_{E,l}(s_\beta(x)-s_\alpha(x),s_\alpha(x))$ are indeed integral affine. By definition, the first Chern class of $\mL(E,l)$ is represented by the \v{C}ech cocycle obtained by differentiating the transition functions for all $\alpha$ and $\beta$. Using the definition of $a_{E,l}$, it follows that $c_1(\mL(E,l))$ is represented by the cocycle
\begin{equation}
\label{equ:representative of Chern class}
(U_{\alpha,\beta}\ni x\mapsto -E(s_\beta(x)-s_\alpha(x))) \in \check C^1(\mathfrak U, \Omega^1_X) \ ,
\end{equation}
where we consider $E$ as a function $\Lambda\to N^*$. To compute what this corresponds to under the identification of $H^1(X,\Omega^1_X)$ with $H^1(X; N^*)\cong \Lambda^*\otimes N^*$, we consider the double complex
\begin{equation*}
(\check C^i(\mathfrak U,\mathcal C^j(X;N^*)), d_{ij}, \partial_{ij}) \ ,
\end{equation*}
where $\mathcal C^i(X;N^*)$ denotes the sheafification of the presheaf
\[
U\mapsto C^i(U;N^*) 
\]
and we set  $C^{-1} (U; N^*)= N^*$ and $\check C^{-1}(\mathfrak U, \mathcal F)= \Gamma(X,\mathcal F)$ for any sheaf $\mathcal F$.
We follow the cocycle of formula \ref{equ:representative of Chern class} through the double complex in the zig-zag from the $(1,-1)$ entry to the $(-1,1)$ entry indicated by the solid arrows in the following diagram: 
\begin{center}
\begin{tikzpicture}[auto]
\matrix[matrix of math nodes, row sep= 5ex, column sep= 3.5em, text height=1.5ex, text depth= .25ex]{
 & [-2.5em]
|(row-1)| 0 &[-1.em]
|(row0)| 0 &[-1.5em]
|(row1)| 0 & [-2.5em]
\\[-2.5ex]
|(col-1)| 0 &
|(-1-1)| N^* &  
|(-10)| C^0(X;N^*) &
|(-11)|	C^1(X;N^*) &
|(colph-1)| \cdots \\
|(col0)| 0 &
|(0-1)| \check C^0(\mathfrak U,(N^*)_X) &
|(00)| \check C^0(\mathfrak U,\mathcal C^0(X;N^*)) &
|(01)| \check C^0(\mathfrak U,\mathcal C^1(X;N^*)) &
|(colph0)| \cdots\\
|(col1)| 0 &
|(1-1)| \check C^1(\mathfrak U,(N^*)_X) &
|(10)| \check C^1(\mathfrak U,\mathcal C^0(X;N^*)) &
|(11)| \check C^1(\mathfrak U,\mathcal C^1(X;N^*)) &
|(colph1)| \cdots \\ [-1.5ex]
 &
|(rowph-1)| \vdots &
|(rowph0)|  \vdots &
|(rowph1)| \vdots  &
\\
};
\begin{scope}[->]
\draw (1-1)--(10);
\draw (00) --(10);
\draw (00) -- (01);
\draw (-11)--(01);
\end{scope}
\begin{scope}[dashed,->,font=\footnotesize]
\draw (col-1)--(-1-1);
\draw (col0) --(0-1);
\draw (col1)--(1-1);

\draw (-1-1)--(-10);
\draw (-10)--(-11);
\draw (0-1)--(00);
\draw (10) -- (11);

\draw (-11)--(colph-1);
\draw (01)--(colph0);
\draw (11)--(colph1);
\draw (row-1)--(-1-1);
\draw (row0)--(-10);
\draw (row1)--(-11);

\draw (-1-1)--(0-1);
\draw(0-1)--(1-1);
\draw (-10)--(00);
\draw (01)--(11);

\draw (1-1)--(rowph-1);
\draw (10)--(rowph0);
\draw (11)--(rowph1);
\end{scope}
\end{tikzpicture}
\end{center}
\vspace{-2ex}
First we apply the differential coming from singular cohomology and obtain
\begin{equation*}
((U_{\alpha,\beta}\xleftarrow{x} \{0\})\mapsto -E(s_\beta(x)-s_\alpha(x)) ) \in  \check C^1(\mathfrak U, \mathcal C^0(X;N^*)) \ .
\end{equation*}
Clearly, this is the image under the differential coming from \v{C}ech cohomology of the cochain
\begin{equation*}
((U_\alpha\xleftarrow{x}\{0\})\mapsto -E(s_\alpha(x)-s(x)) ) \in \check C^0(\mathfrak U, \mathcal C(X; N^*)) \ .
\end{equation*}
Applying the differential of singular cohomology again we obtain 
\begin{equation*}
((U_\alpha\xleftarrow{\sigma}[0,1]) \mapsto -E(s_\alpha(\sigma(1))-s(\sigma(1))-s_\alpha(\sigma(0))+s(\sigma(0))) ) \in \check C^0(\mathfrak U, \mathcal C^1(X; N^*)) \ .
\end{equation*}
This can be lifted to a singular $1$-cochain. Namely, for an arbitrary $1$-simplex $\sigma\colon [0,1]\to X$ we choose a lift $\sigma'\colon  [0,1] \to N_\R$ and then assign to $\sigma$ the value 
\[
-E(\sigma'(1)-s(\sigma(1))-\sigma'(0)+s(\sigma(0))) \ .
\]
This is clearly independent of the choice of $\sigma'$. In particular, if the image of $\sigma$ is contained in $U_\alpha$, we may choose $\sigma'=s_\alpha\circ \sigma$ and obtain the same cocycle on $U_\alpha$ as before. It is also clear that any loop in $X$ which is the image of a path in $N_\R$ from  $0$ to $\lambda\in\Lambda$ is mapped to $-E(\lambda)$ by this $1$-cochain. Therefore, we have $c_1(\mL(E,l))=E$ when identifying $H^{1,1}(X)$ with $\Lambda^*\otimes_\Z N^*$ according to Convention \ref{conv:convention on H11 and bilinear forms}.
\end{proof}

\begin{thm}[Tropical Appell--Humbert Theorem]
\label{thm:Appell-Humbert}
Let $\mL$ be a tropical line bundle on the real torus $X=N_\R/\Lambda$. Then there exists $l\in \Hom(\Lambda,\R)$ and a symmetric form $E$ on $N_\R$ with $E(\Lambda\times N)\subseteq \Z$ such that $\mL\cong \mL(E,l)$. Moreover, if we are given another choice of  $l'\in \Hom(\Lambda,\R)$ and symmetric form $E'$ on $N_\R$ with $E'(\Lambda\times N)\subseteq \Z$, then $\mL\cong \mL(E',l')$ if and only if $E=E'$ and the linear form $(l-l')_\R\colon N_\R\to \R$ has integer values on $N$.
\end{thm}

\begin{proof}
We have already seen in \S\ref{subsec:factors of automorphy} that there exists a tropical factor of automorphy $a\colon \Lambda\times N_\R\to \R$ such that $\mL$ is the line bundle associated to $a(-,-)$. For every $\lambda\in \Lambda$, the function $a(\lambda,-)$ is integral affine, hence its differential $E(\lambda)\coloneqq -da(\lambda,-)$ defines an element in $\Hom(N,\Z)$. Differentiating \eqref{equ:group action condition}, we see that the map $\lambda \mapsto E(\lambda)$ is linear. In other words, $E$ defines a bilinear map on $\Lambda\times N\to \Z$. Therefore, for a suitable function $b\colon \Lambda\to \R$, we have $a(\lambda,x)= -E(\lambda,x) +b(\lambda)$ for all $\lambda\in \Lambda$ and $x\in N_\R$. Plugging this into \eqref{equ:group action condition}, we see that $E(\lambda,\mu)=E(\mu,\lambda)$ for all $\lambda,\mu\in \Lambda$, that is that $E$ is, in fact, symmetric. The tropical factor of automorphy $a-a_{E,0}$ is then a family of constant functions, that is we have
\[
(a-a_{E,0})(\lambda, x)=l(\lambda)
\]
for some function $l\colon \Lambda\to \R$.
Applying \eqref{equ:group action condition} once more we see that $l$ is, in fact, linear. It follows that $a=(a-a_{E,0})+ a_{E,0}=a_{E,l}$. In particular, we have $\mL\cong \mL(E,l)$.

Now assume we are given a second choice of linear function $l'\in \Hom(\Lambda,\R)$ and symmetric form $E'$ on $N_\R$ with $E'(\Lambda\times N)\subseteq \Z$ such that $\mL(E',l')\cong \mL$. We have already seen in \S\ref{subsec:factors of automorphy} that this happens if and only if  $a_{E,l}-a_{E',l'}$ is of the form $a_{0,m_\R\vert_\Lambda}$ for some linear function $m\colon N\to \Z$. By Proposition \ref{prop:chern class recovers symmetric form}, we have
\[
E'=c_1(\mL(E',l'))=c_1(\mL(E,l))= E \ .
\]
Therefore, we have $a_{E,l}-a_{E',l'}= a_{0,l-l'}$ and it follows that $(l-l')_\R$ has integer values on $N$. 
\end{proof}

\begin{rem}
It follows directly from the tropical Appell-Humbert theorem that there is a bijection between the group of all tropical line bundles with trivial first Chern class and $\Lambda^*_\R/N^*$, which is called the dual real torus to $X$ for that reason.
\end{rem}

\subsection{Translations of line bundles}

\begin{prop}
\label{prop:translations of line bundles}
Let $X=N_\R/\Lambda$ be a real torus, let $l\in \Hom(\Lambda, \R)$, and let $E$ be a symmetric bilinear form on $N_\R$ with $E(\Lambda\times N)\subseteq \Z$. Furthermore, let $\pi\colon N_\R\to X$ be the projection, and let $y\in N_\R$. Then we have 
\[
t_{\pi(y)}^*\mL(E,l)\cong \mL(E,l-E(-,y)) \ .
\] 
In particular, if the bilinear form $E$ is nondegenerate and $\mL'$ is any line bundle on $N_\R/\Lambda$ with $c_1(\mL')=E$, then there exists $x\in X$ such that $\mL'\cong t^*_{x}\mL(E,l)$. If, moreover, $E$ restricts to a perfect pairing $\Lambda\times N\to \Z$, then $x$ is unique.
\end{prop}

\begin{proof}
We recall from above that $\mL(E,l)$ can be defined as the quotient of the trivial bundle $N_\R\times \R$ by the $\Lambda$-action given by $\lambda.(x, b) = (x+\lambda, b+ a_{(E,l)}(\lambda,x))$.
Since the morphism $\widetilde t_y \colon N_\R\to N_\R \;\; x\mapsto x+y$ that induces $t_{\pi(y)}$ on the quotient $N_\R/\Lambda$ is $\Lambda$-equivariant, the pull-back $t_{\pi(y)}^*\mL(E,l)$ can be represented as the quotient of 
\[
{\widetilde {t}_y}^* (N_\R\times \R) \cong N_\R\times \R
\]
by the pulled back $\Lambda$-action. The action of $\lambda\in\Lambda$ on $(x,b)$ under the pulled back action is obtained by first applying $\widetilde t_y$ to the first coordinate, yielding $(x+y,b)$, then applying the $\Lambda$-action defined by $a_{(E,l)}$, yielding $(x+y+\lambda, b+a_{(E,l)}(\lambda,x+y))$, and finally applying $\widetilde t_y^{-1}$ to the first coordinate, yielding $(x+\lambda, b+a_{(E,l)}(\lambda,x+y))$. So in total, the pulled back action is given by
\begin{multline*}
\lambda.(x,b)=(x+\lambda,b+a_{(E,l)}(\lambda,x+y)) =\left(x+\lambda, b+ l(\lambda) - E(\lambda, x+y) - \frac 12 E(\lambda,\lambda)\right) \\= \left(x+\lambda, b+ l(\lambda)- E(\lambda,y) - E(\lambda, x)- \frac 12  E(\lambda,\lambda)\right)
= (x+\lambda, b+ a_{(E,l-E(-,y))}) \ .
\end{multline*}
which is precisely the action on the trivial bundle defined by the factor of automorphy $a_{(E,l-E(-,y))}$. This shows that $t^*_{\pi(y)}\mL(E,l) = \mL(E, l-E(-,y))$.

Now assume that $E$ is nondegenerate and that $\mL'$ is any line bundle on $X$ with $c_1(\mL')=E$. By Theorem \ref{thm:Appell-Humbert} and Proposition \ref{prop:chern class recovers symmetric form}, there exists a linear form $l'\colon \Lambda \to \R$ such that $\mL'\cong \mL(E,l')$. Since $E$ is nondegenerate and $\Lambda_\R\cong N_\R$, there exists $\widetilde x\in N_\R$ such that $l-l'= E(-,\widetilde x)$. By what we have shown above, we have 
\[
\mL'\cong \mL(E, l- E(-,\widetilde x))\cong t^*_{\pi(\widetilde x)} \mL(E,l)= t^*_{x} \mL(E,l) \ ,
\]
where $x=\pi(\widetilde x)$. If $x'\in X$ is another point such that $t^*_{x'}\mL(E,l)\cong \mL'$, and $\widetilde x'\in N_\R$ is chosen such that $\pi(\widetilde x')=x'$, then we have
\[
\mL(E, l-E(-,\widetilde x))\cong \mL(E, l- E(-,\widetilde x'))
\]
by what we have shown above. This happens if and only if $E(-,\widetilde x-\widetilde x')$ has integer values on $N$ by Theorem \ref{thm:Appell-Humbert}. If $E$ restricts to a perfect pairing on $\Lambda\times N$, this happens if and only if $\widetilde x-\widetilde x'\in \Lambda$, that is if and only if $x=x'$.
\end{proof}

\begin{rem}
If we call two line bundles on a real torus \emph{tropically equivalent} if they have the same first Chern class, then Proposition \ref{prop:translations of line bundles} shows that two tropical line bundles which are translates of each other are tropically equivalent, with the converse being true if their first Chern class is nondegenerate. This is completely analogous to the situation on complex tori, where two line bundles are \emph{analytically equivalent} if they have the same first Chern class \cite[Proposition 2.5.3]{bila}. If two line bundles on a complex torus are translates of each other, then they are analytically equivalent, with the converse being true if their first Chern class is nondegenerate \cite[Corollary 2.5.4]{bila}.
\end{rem}

\subsection{Rational sections of line bundles.}\label{subsec:rational sections of bundles}
Let $E\colon \Lambda\times N\to \Z$ be bilinear such that $E_\R$ is a symmetric bilinear form on $N_\R$, and let  $l\colon \Lambda \to \R$ be linear. As mentioned above, the tropical line bundle $\mL(E,l)$ on $X$ is a quotient of the trivial bundle $N_\R\times \R$ by the $\Lambda$-action defined by $E$ and $l$. In particular, the global rational sections of $\mL(E,l)$ are precisely those global rational sections of $N_\R\times \R$ that are invariant under the $\Lambda$-action. More precisely, the global rational sections of $\mL(E,l)$ are in bijection with the piecewise linear function $\phi\in \Gamma(N_\R, \mM_{N_\R})$ such that 
\begin{equation}
\label{equ: periodicity of rational functions}
\phi(x+\lambda)= \phi(x)+l(\lambda)-E(\lambda,x)-\frac 12 E(\lambda,\lambda) \ .
\end{equation}
The divisor associated to the section of $\mL(E,l)$ corresponding to $\phi$ is precisely the quotient of $\divv(\phi)$ by the $\Lambda$-action. In particular, this divisor is effective if and only if $\divv(\phi)$ is effective, that is if $\phi$ is concave. Together, concavity and (\ref{equ: periodicity of rational functions}) put strong constraints on $\phi$, or rather its Legendre transform.  In fact, it has been shown in \cite[Theorem 5.4]{MZjacobians} that  if $E$ is a perfect pairing and $E_\R$ is positive definite, these constraints completely determine $\phi$ up to an additive constant. More precisely, $\phi$ is given by
\[
\phi(x)=\min\left\{E(\lambda,x)+\frac 12 E(\lambda,\lambda)- l(\lambda) ~\middle\vert ~\lambda \in \Lambda \right\} + \mathrm{const}
\]
in this case (note that this only differs from the formula in  \cite{MZjacobians} because we are using the ``$\min$''-convention, see Remark \ref{rem:min convention}). By the tropical Appell-Humbert theorem it follows that for every line bundle $\mL$ on $X$ with $c_1(\mL)=E$ there exists a unique effective divisor $D\in \CDiv(X)$ with $\mL(D)=\mL$.

\begin{prop}
\label{prop:effective divisors with positive non-degenerate Chern class}
Let $X=N_\R/\Lambda$ be the real torus associated to a pair of lattices $N$ and $\Lambda\subset N_\R$, and let $D,D'\in \CDiv(X)$  be two effective divisors such that $\cyc[D]=\cyc[D']$ is Poincar\'e dual to $E\in H^{1,1}(X)$ for some perfect pairing $E\colon \Lambda\times N\to \Z$ such that $E_\R$ is a positive definite symmetric bilinear form on $N_\R$, where we identify $H^{1,1}(X)$ with $\Hom(\Lambda,N^*)$ according to Convention \ref{conv:convention on H11 and bilinear forms}. Then there exits a unique $x\in X$ such that $t_x^*D=D'$. 
\end{prop}

\begin{proof}
We have 
\[
\cyc[D]=\cyc(D\cdot [X])=c_1(\mL(D))\frown \cyc[X] , 
\]
so $\cyc[D]$ is Poincar\'e dual to $c_1(\mL(D))$, and similarly $\cyc[D']$ is Poincar\'e dual to $c_1(\mL(D'))$. By assumption, it follow that
\[
c_1(\mL(D))=c_1(\mL(D'))=E \ .
\]
By Proposition \ref{prop:translations of line bundles}, there exists a unique point $x\in X$ such that 
$t_x^*(\mL(D))\cong \mL(t_x^*(D))$ is isomorphic to $\mL(D')$. It follows that the two divisors $t_x^*(D)$ and $D'$ correspond to two concave rational sections of $\mL(D')$. But, since $c_1(\mL(D'))=E$, these two rational sections differ by a constant. Therefore, $D'=t_x^*(D)$.
\end{proof}

\section{Tautological cycles on tropical Jacobians}
\label{sec:geometric cycles}
Classically, the ring of  tautological classes on the Jacobian of an algebraic curve is the smallest subring of its Chow group that contains the image of the curve under the Abel-Jacobi map and is invariant under intersection products, Pontryagin products, translations, and the involution map. We will now introduce the most important tropical tautological cycles on a tropical Jacobian.

Throughout this section, $\Gamma$ will denote a compact connected smooth tropical curve of genus $g$. We will also fix a base point $q\in \Gamma$ with respect to which we define the Abel--Jacobi map.

\subsection{Effective loci and semibreak divisors}
Using the group structure on the Jacobian, the Abel--Jacobi map induces morphisms $\AJ_q^d\colon \Gamma^d\to \Jac(\Gamma)$ for all nonnegative integers $d$. 
\begin{defn}
For every integer $0\leq d \leq g$ we define
\[
\widetilde W_d\coloneqq \AJ_q^d(\Gamma^d) \ .
\]
\end{defn}

Because $\AJ_q^d$ is a proper morphism of \boundaryless{} rational polyhedral spaces, we know that $\tW_d$ is an at most $d$-dimensional \boundaryless{} rational polyhedral subspace of $\Jac(\Gamma)$. By definition, $(\AJ^d_q)_*[\Gamma^d]$ is a tropical $d$-cycle on $\tW_d$. Note that this does not mean that $\tW_d$ has dimension $d$ or that it is pure-dimensional as $(\AJ^d_q)_*[\Gamma^d]$ could be $0$. All we can say a priori is that the support of $(\AJ^d_q)_*[\Gamma^d]$ is precisely the subset of points of $\tW_d$ where the local dimension of $\tW_d$ is equal to $d$.

To show that $\tW_d$ in fact is purely $d$-dimensional we will use the the identification of $\Jac(\Gamma)$ with the $\Pic^0(\Gamma)$ given by the tropical Abel--Jacobi theorem \cite{MZjacobians}. Here, $\Pic(\Gamma)$ denotes the quotient of $\CDiv(\Gamma)$ by the subgroup consisting of all principal divisors, and $\Pic^d(\Gamma)$ denotes the subgroup of $\Pic(\Gamma)$ consisting of the all classes of divisors of degree $d$. The statement of the tropical Abel--Jacobi theorem is that the Abel--Jacobi map $\AJ_q$ induces a bijections $\Pic^d(\Gamma)\to \Jac(\Gamma)$ for $d=0$, and hence for any $d$. If $W_d$ denotes the preimage of $\tW_d$ in $\Pic^d(\Gamma)$ under the bijection $\Pic^d(\Gamma)\to \Jac(\Gamma)$, then $W_d$ is precisely the set of the classes of effective divisors of degree $d$. In particular $W_d$ is independent of the base point $q$. Together with L.\ T\'othm\'er\'esz, we have proved the following theorem.

\begin{thm}[{\cite[Theorem 8.3]{semibreak}}]
\label{thm:pure dimensionality}
The subset $W_d$ of $\Pic^d(\Gamma)$ is purely $d$-dimensional.
\end{thm}

It follows immediately that  $\widetilde W_d$ is purely $d$-dimensional as well, and hence that the tropical cycle $(\AJ^d_q)_*[\Gamma^d]$ has support $\widetilde W_d$. We will now show that $\tW_d$ has a fundamental cycle $[\tW_d]$ which we will relate to $(\AJ^d_q)_*[\Gamma^d]$. To do this, we will need the notion of a break and semibreak divisors.  A \emph{break divisor} on $\Gamma$ is an effective divisor $B$ such that there exist $g$ open edge segments $e_1,\ldots,e_g\subseteq \Gamma$ and points $q_i\in \overline e_i$ such that $\Gamma\setminus \bigcup_i e_i$ is contractible and $B=\sum_i (q_i)$. A \emph{semibreak divisor} is an effective divisor that is dominated by a break divisor, that is an effective divisor $D$ such that there exists an effective divisor $E$ for which $D+E$ is a break divisor (cf. \cite{semibreak}).

\begin{prop}
\label{prop:i! times W_i}
Let $0\leq d\leq g$. Then $\tW_d$ has a fundamental cycle $[\widetilde W_d]$, and the equality
\[
(\AJ^d_q)_*[\Gamma^d]= d! [\widetilde W_d] 
\]
hold in $Z_*(\Jac(\Gamma))$. 
\end{prop}

\begin{proof}
It suffices to show that $(\AJ_q^d)_*[\Gamma^d]$ has weight $d!$ on all components of $\tW_d^\reg$. Indeed, if that is the case then $\frac 1{d!}(\AJ_q)_*[\Gamma^d]$ is a tropical cycle with support $\tW_d$ and weight $1$ on all components of $\tW_d^\reg$. But this implies that $\tW_d$ has a fundamental cycle and that $(\AJ^d_q)_*[\Gamma^d]=d![\tW_d]$. 

By the definition of the push-forward, we now have to show that for any $x\in \tW_d$ such that $(\AJ_q^d)^{-1}\{x\}$ is finite and contained in $(\Gamma^d)^\reg$, the value of $(\AJ^d_q)_*[\Gamma^d]$ at $x$  is $d!$. 
Let $\sigma$ be a component of $(\Gamma^d)^\reg$. Then there exist open edges $e_1,\ldots,e_d$ of $\Gamma$ such that $\sigma=e_1\times\ldots\times e_d$. We choose an orientation on each of these $d$ edges. This determines a unique primitive tangent vector $\eta_k$ on each each edge $e_k$.
These $d$ tangent vectors form a basis of the integral tangent space of the product $e_1\times \ldots\times e_d$.  As already noted in \S\ref{sec:tropical Jacobians}, the image of $\eta_k$ in the tangent space $\Omega_\Z(\Gamma)^*$ of $\Jac(\Gamma)$ is given by 
\[
(d\AJ_q)(\eta_k)\colon\Omega_\Z(\Gamma)\to \Z,\;\;\omega\mapsto \langle \omega|_{e_k}, \eta_k\rangle \ .
\]
If we identify $\Omega_\Z(\Gamma)$ with integral flows on $\Gamma$, as explained in Remark \ref{rem:flows}, then $(d\AJ_q)(\eta_k)$ is the map assigning to an integral flow $\omega$ on $\Gamma$ its flow on $e_k$ in the direction specified by the chosen orientation.
Because $\AJ_q^d$ is defined as the $d$-fold sum of $\AJ_q$, we have $(d\AJ_q^d)(\eta_k)=(d\AJ_q)(\eta_k)$.
In particular, if $e_k=e_l$ for $k\neq l$, then $(d\AJ_q^d)(\eta_k)= (d\AJ_q^d) (\eta_l)$ which means that $\AJ_q^d$ is not injective on $\sigma$ and $x\notin \AJ_q^d(\sigma)$. We may thus assume that all $e_k$ are distinct. If $\Gamma\setminus \bigcup e_k$ is disconnected, then there exists an $1\leq l \leq d$ such that $\Gamma\setminus \cup_{k=1}^l e_k$ has precisely two components $C_1$ and $C_2$. For $1\leq k\leq l$ let $\alpha_k$ be equal $1$ if $e_k$ is oriented such that it leads from $C_1$ to $C_2$, and let $\alpha_k$ be equal to $-1$ if it is oriented the other way. Since the total flow from $C_1$ to $C_2$ in any integral flow on $\Gamma$ is $0$, we have
\[
\sum_{k=1}^l \alpha_k  (d\AJ_q^d)(\eta_k) = 0 \ ,
\]
which means that $d\AJ_q^d$ is not injective on the tangent spaces of $\sigma$. Therefore, $\AJ_q^d$ is not injective on $\sigma$ and again $x\notin \AJ_q^d(\sigma)$. If $\Gamma\setminus \bigcup e_k$ is connected, then for each $1\leq k\leq d$ there is a simple closed loop in $\Gamma$ that passes through $e_k$ but not through $e_l$ for $l\neq k$. It follows that for every assignment of values $f\colon\{1,\ldots, d\}\to \Z$ there is an integral flow $\omega\in \Omega_\Z(\Gamma)$ whose flow on $e_k$ is $f(k)$. This implies that the vectors $(d\AJ_q^d)(\eta_1),\ldots, (d\AJ_q^d)(\eta_i)$ span a saturated rank-$i$ sublattice of $\Omega_\Z(\Gamma)^*$. Therefore, every point of $(\AJ_q^d)^{-1}\{x\}\cap \sigma$ contributes to the weight of $(\AJ^d_q)_*[\Gamma^d]$ with multiplicity one, and by \cite[Lemma 8.1]{semibreak} there is at most one of these points. In fact, if  $(\AJ_q^d)^{-1}\{x\}\cap \sigma$ is nonempty, then \cite[Lemma 8.1]{semibreak} tells us that all other components $\sigma'$ of $(\Gamma^d)^\reg$ with $(\AJ_q^d)^{-1}\{x\}\cap \sigma'\neq\emptyset$ are obtained from $\sigma$ via a permutation of coordinates. As there are exactly $d!$ of these permutations, the weight at $x$ is $d!$, finishing the proof.
\end{proof}

As an immediate consequence of Proposition \ref{prop:i! times W_i} we obtain the following corollary.

\begin{cor}
\label{cor:Wi as Pontryagin product}
The equality of tropical cycles $\star_{k=1}^d [\tW_1]=d! [\tW_d]$ holds in $Z_*(\Jac(\Gamma))$.
\end{cor}

\begin{proof}
This follows directly from the formulas for $[\tW_d]$ and $[\tW_1]$ given in Proposition \ref{prop:i! times W_i}, and the fact that $\AJ_q^d$ is the $d$-fold sum of $\AJ_q$.
\end{proof}

We have a morphism $H_1(\Gamma;\Z)\to H_1(\Jac(\Gamma);\Z)$ induced by the (continuous) Abel--Jacobi map. As noticed in \S\ref{sec:homology of Abelian varieties}, there is a natural identification 
\[
H_1(\Jac(\Gamma);\Z)\cong H_1(\Gamma;\Z)
\]
 coming from the fact that $\Jac(\Gamma)=\Omega_\R(\Gamma)^*/H_1(\Gamma;\Z)$ is defined by taking a quotient of a real vector space by $H_1(\Gamma;\Z)$. 

\begin{lemma}
\label{lem:pushforward is identity}
The morphism
\[
(\AJ_q)_*\colon H_1(\Gamma;\Z)\to H_1(\Jac(\Gamma);\Z)\cong H_1(\Gamma;\Z)
\]
is the identity.
\end{lemma}

\begin{proof}
Let $\alpha$ be a cycle on $\Gamma$ representing a class in $H_1(\Gamma;\Z)$. We need to show that $(\AJ_q)_*[\alpha]=[\alpha]$. By the Hurewicz theorem, we may assume that it is represented by a loop $\gamma\colon [0,1]\to \Gamma$ starting and ending at the base point $q$. By the definition of the Abel--Jacobi map, the path
\begin{equation*}
\widetilde\gamma \colon [0,1]\to \Omega_\R(\Gamma)^*,\;\; t\mapsto \left( \omega \mapsto \int_{ \gamma|_{[0,t]}} \omega \right)
\end{equation*}
lifts the composite $\AJ_q\circ \gamma$. Therefore, $(\AJ_q)_*\gamma\in H_1(\Jac(\Gamma);\Z)$ is identified with the element $\widetilde\gamma(1)-\widetilde\gamma(0)=\widetilde\gamma(1)\in H_1(\Gamma;\Z)$. But this is equal to the image of $\gamma$ under the embedding $H_1(\Gamma;\Z)\hookrightarrow \Omega_\R(\Gamma)^*$. 
\end{proof}

\subsection{The tropical Riemann theta divisor} \label{subsec:Theta divisor}

Recall from \S\ref{sec:tropical Jacobians} that the tropical Jacobian $\Jac(\Gamma)=\Omega_\R(\Gamma)^*/ H_1(\Gamma;\Z)$ of a smooth tropical curve $\Gamma$ comes equipped with a positive definite symmetric form $Q$ on its universal cover $\Omega_\R(\Gamma)^*$ which restricts to a perfect pairing $\Omega_\Z(\Gamma)^*\times H_1(\Gamma;\Z)\to \Z$. By Proposition \ref{prop:chern class recovers symmetric form}, the first Chern class of the line bundle $\mL(Q,0)$ is given by $Q$. As explained in \S\ref{subsec:rational sections of bundles}, this implies that $\mL(Q,0)$ has, up to an additive constant, a unique concave rational section, the \emph{Riemann theta function}, which defines a unique effective divisor $\Theta\in\CDiv(\Jac(\Gamma))$ with $\mL(\Theta)= \mL(Q,0)$. For further details about the Riemann theta function see \cite{MZjacobians}, and see \cite{FRSS} for the connection to the non-archimedean Riemann theta function.

\begin{defn}
The unique effective divisor $\Theta\in \CDiv(\Jac(\Gamma))$ with $\mL(\Theta)=\mL(Q,0)$ is called the \emph{tropical Riemann theta divisor} on $\Jac(\Gamma)$.
\end{defn}

Note that by construction, we have $c_1(\mL(\Theta))=Q$.

\begin{example}
Figure \ref{fig:Theta and W1} shows the $\Theta$-divisor for the curve $\Gamma$ from Example \ref{example:Theta and W1}. It is the image in $\Jac(\Gamma)$ of the boundaries of the Voronoi cells of the lattice points $H_1(\Gamma;\Z)$ in $\Omega_\R(\Gamma)^*$ with respect to the metric defined by $Q$.
\end{example}

\section{The tropical Poincar\'e formula}
\label{sec:the formula}

We are finally in a position to prove the Poincar\'e formula. Our strategy is to give explicit formulas for both sides of the equation. More precisely, we will introduce coordinates on the tropical homology groups of the tropical Jacobian, and will compare the coefficients of both sides of the equation in these coordinates. Throughout this section, $\Gamma$ will denote a compact and connected smooth tropical curve of genus $g$, and  $e_1,\ldots, e_g$ will denote distinct open edges of $\Gamma$ such that $\Gamma\setminus (\bigcup_k e_k)$ is contractible. Furthermore, we will assume that we have chosen an orientation on each of the edges $e_1,\ldots, e_g$.

\subsection{Bases for the tropical (co)homology of $\Jac(\Gamma)$}
\label{subsec:generators for homology}
Recall from \S\ref{sec:homology of Abelian varieties} that there is an isomorphism of rings
\[
H_{*,*}(\Jac(\Gamma))\cong \bigwedge H_1(\Gamma;\Z) \otimes \bigwedge \Omega_\Z(\Gamma)^* \ ,
\]
where the ring structure on the left side is given by the Pontryagin product. 
Using this isomorphism, a choice of bases for $H_1(\Gamma;\Z)$ and $\Omega_\Z(\Gamma)^*$ will induce a basis for $H_{*,*}(\Jac(\Gamma))$. We will use our choice of open edges $e_1,\ldots, e_g$ to define bases for these lattices. Let $1\leq k\leq g$. The orientation on $e_k$ defines a start and an end point for $e_k$. Since $T$ is contractible and therefore a tree, there is a path in $T$ from the end to the start point of $e_k$, and this path is unique up to homotopy. Together with any path in $\overline e_k$ from its start to its end point, this defines a \emph{fundamental circuit} $c_k\in H_1(\Gamma;\Z)$ that traverses $e_k$ but is disjoint from $e_l$ for $l\neq k$. It is well known, and straightforward to check, that the the fundamental circuits $c_1,\ldots, c_g$ form a basis of $H_1(\Gamma;\Z)$. 

To obtain a basis for $\Omega_\Z(\Gamma)^*$, let $\eta_k$ denote the primitive tangent vector on $e_k$ in the direction specified by the orientation, and let $\vf{k}=(d\AJ_q)(\eta_k)$. As we observed in \S\ref{sec:tropical Jacobians}, $\vf{k}$ can be described as the morphism $\Omega_\Z(\Gamma)\to \Z$ assigning to an integral flow on $\Gamma$ its flow through $e_k$ in the direction specified by the orientation. By definition of the bilinear from $Q$ on $\Omega_\R(\Gamma)$, we have $Q(c_k,\vf l)=1$ if $k=l$ and $Q(c_k,\vf l)=0$ if $k\neq l$, that is $\vf 1,\ldots,\vf g$ is dual to the basis $c_1,\ldots,c_g$ with respect to $Q$. We noticed in \S\ref{sec:tropical Jacobians} that $\Omega_\Z(\Gamma)^*$ is precisely the set of vectors in $\Omega_\R(\Gamma)^*$ that have integral pairing with respect to $Q$ with all elements of $H_1(\Gamma;\Z)$. It follows directly that $\vf 1,\ldots,\vf g$ is a basis for $\Omega^*_\Z(\Gamma)$.

Similarly, by the isomorphism
\[
H^{*,*}(\Jac(\Gamma))\cong \bigwedge H_1(\Gamma;\Z)^* \otimes \bigwedge \Omega_\Z(\Gamma)
\]
of rings discussed in \S\ref{sec:homology of Abelian varieties}, bases for $H_1(\Gamma;\Z)^*$ and $\Omega_\Z(\Gamma)$  induce a basis for $H^{*,*}(\Jac(\Gamma))$. The bases we will use for these lattices are the dual bases $(c_k^*)_k$ and $(\vf k^*)_k$ to the bases $(c_k)_k$ and $(\vf k)_k$.

Note that both $H_{*,*}(\Jac(\Gamma))$ and $H^{*,*}(\Jac(\Gamma))$ are tensor products of skew commutative graded rings. We will use the following notation for elements of special form in groups of this type.
\begin{notation}
\label{notation:products in tensor of skew algebras}
Let $R_1$ and $R_2$ be two skew-commutative graded rings, let $J$ be a finite set, and let $a\colon J\to R_1$ and $b\colon J\to R_2$  be maps such that for every $j\in J$ the elements $a(j)$ and $b(j)$ are homogeneous of the same degree. Then for any injective map $\sigma\colon\{1,\ldots, k\} \to J$, the element
\[
\prod_{l=1}^k a(\sigma(l)) \otimes \prod_{l=1}^k b(\sigma(l)) 
\]
of $R_1\otimes_\Z R_2$ only depends on the image $I\coloneqq\sigma(\{1,\ldots, k\})$. We denote it by
\[
\prod_{i\in I} a(i) \otimes \prod_{i\in I} b(i) \ .
\]
\end{notation}

\subsection{Cycle classes of tautological cycles}
\label{subsec:cycle classes of tautological cycles}

\begin{prop}
\label{prop:formula for homology class of W1}
We have 
\[
\cyc[\tW_1] =\sum_{k=1}^g c_k\otimes \vf k \ .
\]
\end{prop}

\begin{proof}
Choose an orientation for every edge $e$ of $\Gamma$ that coincides with the orientation we have already chosen if $e=e_k$ for some $k$. Let $\eta_e$ the primitive tangent vector of $e$ in the direction specified by the orientation, and let $\vf e=(d\AJ_q)(\eta_e)$. By construction, we have $\vf{e_k}=\vf k$ for all $1\leq k\leq g$. It follows immediately from the definition of the tropical cycle class map and Theorem \ref{thm:cycle class map commutes with operations} that $\cyc[\tW_1]$ is represented by the $(1,1)$-cycle
\begin{equation*}
\sum_{e\in E(\Gamma)} (\AJ_q)_*(\overline e)\otimes \vf e \in C_{1,1}(X) \ ,
\end{equation*} 
where we view the oriented closed edge $\overline e$ as a singular $1$-simplex by choosing a parametrization compatible with the given orientation. Using that the $c_k$ and the $\vf k$ form dual bases with respect to the bilinear form $Q$, we see that the above equals
\begin{equation*}
\sum _{e\in E(\Gamma)}(\AJ_q)_*(\overline e)\otimes \left ( \sum_{i=1}^g Q(c_i, \vf e)\cdot \vf i\right) =
\sum_{i=1}^g \left ( \sum_{e\in E(\Gamma)} Q( c_i,\vf e)\cdot (\AJ_q)_*(\overline e) \right ) \otimes \vf i \ .
\end{equation*}
Since $Q(c_i,\vf e)$ is $1$ whenever $e$ is on the loop $c_i$, and $0$ otherwise, we have 
\begin{equation*}
\sum_{e\in E(\Gamma)} Q(c_i,\vf e) (\AJ_q)_*(\overline e)=(\AJ_q)_*c_i \ ,
\end{equation*}
which is equal to $c_i$ by Lemma \ref{lem:pushforward is identity}. This finishes the proof.
\end{proof}

\begin{rem}
It follows immediately from Proposition \ref{prop:formula for homology class of W1} that the expression
\[
\sum_k c_k\otimes \vf k \in H_1(\Gamma;\Z)\otimes \Omega_\Z(\Gamma)^*
\]
is independent of the choice of spanning tree used to define the elements $c_k$ and $\vf k$. On a closer look, it turns out that this independence is more of a feature of linear algebra than a feature of spanning trees. To see this, we observe that the natural isomorphism $H_1(\Gamma;\Z)\cong \Omega_\Z(\Gamma)$ identifies the basis $(\vf k)_k$ with the dual basis of $(c_k)_k$. Therefore, $\sum_k c_k\otimes \vf k$ is identified with the identity endomorphism on $H_1(\Gamma;\Z)$ under the composite
\[
H_1(\Gamma;\Z)\otimes \Omega_\Z(\Gamma)^*\cong H_1(\Gamma;\Z)\otimes H_1(\Gamma;\Z)^*\cong \End(H_1(\Gamma;\Z)) \ ,
\]
which is an invariant of $H_1(\Gamma;\Z)$ rather than of $\Gamma$.
\end{rem}

\begin{lemma}
\label{lem:formula for homology class of Wi}
We have
\[
\cyc[\tW_d]= 
\sum_{\substack{I\subseteq \{1,\ldots, g\} \\|I|=d}} \bigwedge_{k\in I} c_k \otimes \bigwedge_{k\in I} \vf{k} \ .
\]
\end{lemma} 

\begin{proof}
Using Proposition \ref{prop:i! times W_i} and Proposition \ref{prop:Pontryagin product commutes with cycle map} we obtain
\[
d!\cyc[\tW_d] = d!\cyc \left(\bigstar_{k=1}^d [\tW_1]\right) =d!\bigstar_{k=1}^d \cyc[\tW_1] \ .
\]
By Proposition \ref{prop:formula for homology class of W1}, this equals
\[
\bigstar_{k=1}^d\left( \sum_{l=1}^g c_l\otimes \vf{l} \right) \ . 	
\]
Using the description of the Pontryagin product from \S\ref{sec:homology of Abelian varieties}, we can rewrite this as
\[
\sum_\sigma \bigwedge_{k=1}^d c_{\sigma(k)} \otimes \bigwedge_{k=1}^d \vf{\sigma(k)} \ ,
\]
where the sum is over all maps $\sigma\colon \{1,\ldots,d\}\to \{1,\ldots,g\}$. Since $\bigwedge \Omega_\Z(\Gamma)$  is skew-commutative, only an injective $\sigma$ would contribute to the sum. If $I$ is the image of an injective $\sigma$ then, using our Notation \ref{notation:products in tensor of skew algebras}, we have 
\[
\bigwedge_{k=1}^d c_{\sigma(k)} \otimes \bigwedge_{k=1}^d \vf{\sigma(k)} = \bigwedge_{k\in I} c_k \otimes \bigwedge_{k\in I} \vf{k} \ .
\]
Since the map $\sigma\mapsto \sigma(I)$, for injective $\sigma\colon \{1,\ldots,d\}\to \{1,\ldots,g\}$, is $d!$-to-$1$, we obtain
\[
d!\cyc[\tW_d]=\sum_\sigma \bigwedge_{k=1}^d c_{\sigma(k)} \otimes \bigwedge_{k=1}^d \vf{\sigma(k)} = d!\sum_{\substack{I\subseteq \{1,\ldots, g\} \\|I|=d}} \bigwedge_{k\in I} c_k \otimes \bigwedge_{k\in I} \vf{k} \ .
\]
The result follows after dividing both sides by $d!$. This division is allowed because the tropical homology groups of $\Jac(\Gamma)$ are torsion-free.
\end{proof}

\subsection{Tropical cycle classes of powers of the theta divisor}
\label{subsec:powers of theta}

\begin{lemma}
\label{lem:formula for Chern class of theta divisor}
We have
\[
c_1(\mL(\Theta))=\sum_{i=1}^g c_i^*\otimes \vf i^*.
\]
\end{lemma}

\begin{proof}
As already observed in \S\ref{subsec:Theta divisor}, we have $c_1(\mL(\Theta))=Q$, where we identify
\[
H^{1,1}(\Jac(\Gamma))\cong H^1(\Gamma;\Z) \otimes \Omega_\Z(\Gamma)
\]
with $\Hom(H_1(\Gamma;\Z)\otimes \Omega_\Z(\Gamma),\Z)$.
Because  $(c_k)_k$ and $(\vf k)_k$ are dual bases with respect to $Q$, the assertion follows.
\end{proof}

\begin{lemma}
\label{lem:Poincare duality for some monomial classes}
Let $I\subseteq \{1,\ldots, g\}$ with $|I|=d$. Then 
\[
\bigwedge_{k\in I} c_k^* \otimes \bigwedge_{k\in I} \vf{k}^*\in  \bigwedge\nolimits^d H_1(\Gamma,\Z)^* \otimes \bigwedge\nolimits^d \Omega_\Z(\Gamma) \cong H^{d,d}(\Jac(\Gamma))
\] 
is Poincar\'e dual to 
\[
\bigwedge_{k\in \{1,\ldots,g\} \setminus I} c_k \otimes \bigwedge_{k\in \{1,\ldots,g\}\setminus I} \vf{k}\in  H_{g-d}(\Jac(\Gamma),\Z)\otimes \bigwedge\nolimits^{g-d}\Omega_\Z(\Gamma)^* \cong H_{g-d,g-d}(\Jac(\Gamma))\ .
\]
\end{lemma}

\begin{proof}
Since $\Jac(\Gamma)=\tW_g$, we have
\begin{equation}
\label{equ:formula for fudamental calss}
\cyc[\Jac(\Gamma)]=\left(\bigwedge_{1\leq k \leq g} c_k\right)\otimes \left(\bigwedge_{1\leq k\leq g}\vf{k}\right)
\end{equation}
by Lemma \ref{lem:formula for homology class of Wi}.

Note that for every $\alpha\in H_1(\Jac(\Gamma);\Z)^*$, $x\in \bigwedge ^iH_1(\Jac(\Gamma);\Z)$, and $y\in \bigwedge ^j H_1(\Jac(\Gamma);\Z)$ we have
\[
\alpha\,\lrcorner\,(x\wedge y)= (\alpha \,\lrcorner\, x)\wedge x +(-1)^i  x\wedge (\alpha\,\lrcorner\, y)
\]
by the properties of the interior product (see  \cite[Proposition A 2.8]{Eisenbud}). 
Similarly for every $a\in \Omega_\Z(\Gamma)$, $b\in \bigwedge^i \Omega_\Z(\Gamma)^*$, and $c\in \bigwedge^j\Omega_\Z(\Gamma)^*$ we have 
\[
a\,\lrcorner\, (b\wedge c)= (a \,\lrcorner\, b)\wedge c +(-1)^i  b\wedge (a\,\lrcorner\, c) \ .
\]
Using induction, we conclude that
\[
\bigwedge_{k\in I}c_k^* \frown \bigwedge_{k\in \{1,\ldots,g\}} c_k = \pm \bigwedge_{k\in \{1,\ldots,g\}\setminus I}c_k
\]
and similarly
\[
\bigwedge_{k\in I}\vf{k}^*\frown \bigwedge_{k\in \{1,\ldots,g\}} \vf{k} =\pm \bigwedge_{k\in\{1,\ldots,g\}\setminus I} \vf{k} \ , 
\]
with the sign being the same on the right-hand sides of the two equations as long as we order the sets $I$ and $\{1,\ldots, g\}$ consistently in both equations.
Combining these identities  with the expression (\ref{equ:formula for fudamental calss}) for the fundamental class of $\Jac(\Gamma)$ and the identity (\ref{equ:identity for cap product on real tori}), it follows that 
\[
\left( \bigwedge _{k\in I }c_k^* \otimes \bigwedge_{k\in I} \vf{k}^*\right) \frown [\Jac(\Gamma)]  = \bigwedge_{k\in \{1,\ldots,g\}\setminus I} c_k \otimes \bigwedge_{k\in \{1,\ldots,g\}\setminus I} \vf{k} \ , 
\]
which is precisely what we needed to show.
\end{proof}

The following result is the tropical analogue of \cite[Theorem 4.10.4]{bila}.
\begin{lemma}
\label{lem:homology class of power of Theta}
We have 
\[
\cyc([\Theta]^{g-d}) =(g-d)!\sum_{\substack {I\subseteq \{1,\ldots, g\}\\ |I|=d}} \bigwedge_{k\in I} c_k \otimes \bigwedge_{k\in I} \vf{k} \ . 
\]
\end{lemma}

\begin{proof}
Since intersections with divisors is compatible with the tropical cycle class map by Theorem \ref{thm:cycle class map commutes with operations}, we have 
\[
\cyc([\Theta]^{g-d})= c_1(\mL(\Theta))^{g-d}\frown [\Jac(\Gamma)] \ ,
\]
that is $\cyc([\Theta]^{g-d})$ is Poincar\'e dual to $c_1(\mL(\Theta))^{g-d}$.
By Lemma \ref{lem:formula for Chern class of theta divisor} we know that 
\[
c_1(\mL(\Theta))= \sum_{i=1}^g c_i^*\otimes \vf {e_i}^* \ .
\]
With the description of the cap-product on $H^{*,*}(X)$ given in \S\ref{sec:homology of Abelian varieties}, we obtain
\begin{equation*}
c_1(\mL(\Theta))^{g-d} =  
(g-d)!\sum_{\substack{I\subseteq \{1,\ldots, g\} \\ |I|= g-d}}  \bigwedge_{k\in I} c_k^* \otimes  \bigwedge_{k\in I} \vf{ e_k}^* 
\end{equation*}
similar as in the proof of Lemma \ref{lem:formula for homology class of Wi}.
Applying Lemma \ref{lem:Poincare duality for some monomial classes} finishes the proof.
\end{proof}

\subsection{The proof of the tropical Poincar\'e formula}
\begin{thm}
\label{thm:Poincare formula}
The Poincar\'e formula holds tropically, that is we have
\[
(g-d)![\tW_d] \sim_{\mathrm{hom}} [\Theta]^{g-d} \ .
\]
\end{thm}

\begin{rem}
The Poincar\'e formula is more commonly expressed as
\[
[\tW_d] \sim_{\mathrm{hom}} \frac{1}{(g-d)!} [\Theta]^{g-d} \ ,
\]
where we the right side is defined after an extension of scalars to $\Q$. Because the tropical homology groups of Jacobians are torsion-free, this is indeed an equivalent expression of the formula.
\end{rem}

\begin{proof}
By Lemma \ref{lem:formula for homology class of Wi} we have
\[
\cyc[\tW_d] = \sum_{\substack{I\subseteq \{1,\ldots, g\} \\|I|=d}} \bigwedge_{k\in I} c_k \otimes \bigwedge_{k\in I} \vf{k} \ . 	
\]
On the other hand, by Lemma \ref{lem:homology class of power of Theta} we have
\[
\cyc([\Theta]^{g-d}) =(g-d)!\sum_{\substack {I\subseteq \{1,\ldots, g\}\\ |I|=d}} \bigwedge_{k\in I} c_k \otimes \bigwedge_{k\in I} \vf{k} \ .
\]
It follows immediately that
\[
\cyc((g-d)![\tW_d])=\cyc([\Theta]^{g-d}) \ ,
\]
which is equivalent to saying that $[\tW_d]$ and $[\Theta]^{g-d}$ are homologically equivalent.
\end{proof}

\begin{cor}
\label{cor:poincare up to numerical equivalence}
We have 
\[
(g-d)![\tW_d] \sim_{\mathrm{num}} [\Theta]^{g-d} \ .
\]
\end{cor}

\begin{proof}
This follows directly from Theorem~\ref{thm:Poincare formula} and Proposition \ref{prop:homological equivalence implies numerical equivalence}.
\end{proof}

\begin{rem}
\label{rem:reduction to boundaryless case}
We have proved Theorem \ref{thm:Poincare formula} under the assumption that the smooth tropical curve is \boundaryless{}.  If $\Gamma$ is a compact and connected smooth tropical curve with boundary as described in Remark \ref{rem:curves with boundary}, then the Poincar\'e formula holds as well, and the proof in this seemingly more general case can easily be reduced to the \boundaryless{} case. Namely, if $\Gamma'$ denotes the \boundaryless{} smooth tropical curve obtained from $\Gamma$ by removing the leaves from $\Gamma$, then $\Gamma$ and $\Gamma'$ have identical Jacobians, and their theta divisors coincide by definition. Furthermore, the Abel-Jacobi map associated to $\Gamma'$ contracts all the leaves of $\Gamma'$, so that the loci $\tW_d$ associated to $\Gamma$ and $\Gamma'$ coincide as well.
\end{rem}

\subsection{Consequences of the Poincar\'e formula}
The tropical Poincar\'e formula has some interesting immediate consequences. One of them is a tropical version of Riemann's theorem. The statement has appeared before \cite{MZjacobians}, with a different (combinatorial) proof. To state the theorem, recall from \S\ref{sec:geometric cycles} that the Abel-Jacobi map induces a bijection $\Pic^0(\Gamma)\to \Jac(\Gamma)$. Because all contributions from the chosen base point $q$ cancel in degree $0$, this bijection is independent of all choices. In particular, we can view $\Theta$ as a divisor on $\Pic^0(\Gamma)$ in a natural way. Also recall from \S\ref{sec:geometric cycles} that while $\tW_d\subseteq \Jac(\Gamma)$ depends on $q$, the image $W_d$ of $\Gamma^d$ in $\Pic^d(\Gamma)$ does not.

\begin{cor}[Tropical Riemann's Theorem](cf.\ \cite[Corollary 8.6]{MZjacobians})
\label{cor:Riemman's Theorem}
There exists a unique  $\mu\in \Pic^{g-1}(\Gamma)$ such that
\[
[W_{g-1}]=\mu+[\Theta] \ ,
\]
where we consider $[\Theta]$ as a tropical cycle in $\Pic^0(\Gamma)$.
\end{cor}

\begin{proof}
It suffices to show that there exists a unique $\mu\in \Jac(\Gamma)$ such that $[\tW_{g-1}]=(t_{\mu})_*[\Theta]$ when considering $\Theta$ as a divisor on $\Jac(\Gamma)$. Since $[\tW_{g-1}]$ is a codimension-$1$ tropical cycles on the tropical manifold $\Jac(\Gamma)$, we can view $\tW_{g-1}$ as a tropical Cartier divisor as well (see \S\ref{subsec:tropical cycles})).  Applying the Poincar\'e formula (Theorem \ref{thm:Poincare formula}) with $d=g-1$ yields
\[
\cyc[\tW_{g-1}]=\cyc[\Theta] \ .
\]
By definition of $\Theta$, the cycle class $\cyc[\Theta]$ is Poincar\'e dual to the element in $H^{1,1}(\Jac(\Gamma))$ corresponding to the linear form $Q$. 
As $Q$ restricts to a perfect pairing $H_1(\Gamma;\Z)\times \Omega_\Z(\Gamma)^*\to \Z$, Proposition \ref{prop:effective divisors with positive non-degenerate Chern class} applies and there a unique $\mu\in \Jac(\Gamma)$ such that $t_{\mu}^*\tW_{g-1}=\Theta$. This is, of course, equivalent to the equality $(t_{\mu})_*[\tW_{g-1}]=[\Theta]$.
\end{proof}

\begin{cor}
\label{cor:degree of tWd times tWg-d}
For every $0\leq d\leq g$, we have
\[
\int_{\Jac(\Gamma)}[\tW_d]\cdot [\tW_{g-d}] = {g\choose d} \ .
\]
\end{cor}

\begin{rem}
In the special case $d=1$ we recover the formula
\[
\int_{\Jac(\Gamma)}[\tW_1]\cdot [\Theta]= g 
\]
stated in \cite[Theorem 6.5]{MZjacobians}. Also note that the intersection product $[\tW_d]\cdot [\tW_{g-d}]$ is effective since one can locally apply the fan displacement rule.
\end{rem}

\begin{proof}
We apply Poincar\'e formula (Theorem \ref{thm:Poincare formula}) three times, and obtain a chain of equalities
\[
[\tW_d]\cdot [\tW_{g-d}] = \frac{[\Theta]^g}{d!(g-d)!}= {g\choose d} [\tW_0]
\]
that hold modulo homological equivalence. Taking the degree yields the result.
\end{proof}

\begin{cor}
\label{cor:Theta to the g}
We have
\[
\int_{\Jac(\Gamma)} [\Theta]^g = g! \ .
\]
\end{cor}

\begin{proof}
By the tropical Poincar\'e formula (Theorem \ref{thm:Poincare formula}), we have
\[
\int_{\Jac(\Gamma)}[\Theta]^g= g! \int_{\Jac(\Gamma)} [\tW_0]= g! \ .
\]
\end{proof}

\begin{rem}
Classically, the statement of Corollary \ref{cor:Theta to the g} also follows from the geometric Riemann-Roch theorem for Abelian varieties \cite[Theorem 3.6.3]{bila}. Tropically, it is also possible to prove the statement using the duality of Voronoi and Delaunay decompositions.
\end{rem}



\end{document}